\UseRawInputEncoding
\documentclass[11pt]{article}
\usepackage{geometry}
\usepackage{bbm}
\usepackage{amsfonts}
\usepackage{pifont}
\usepackage{bbding}
\usepackage{latexsym}
\usepackage{mathrsfs}
\usepackage{amsmath}
\usepackage{cite}
\usepackage{amsthm}
\usepackage{amssymb}
\usepackage{graphicx}

\newtheorem{theorem}{\indent Theorem}[section]
\newtheorem{corollary}{\indent Corollary}[section]
\newtheorem{proposition}{\indent Proposition}[section]
\newtheorem{definition}{\indent Definition}[section]
\newtheorem{lemma}{\indent Lemma}[section]
\newtheorem{remark}{\indent Remark}[section]
\numberwithin{equation}{section}

\date{}
\begin{document}

\begin{center}
{\Large \bf Multiple positive solutions for the fractional Schr\"{o}dinger-Poisson systems involving critical nonlinearities with potential\footnote{This work is supported by NSF DMS-1804497.} }\\
\vspace{0.5cm} {Haining Fan}\\ \vspace{6pt}
{\footnotesize {\em{ Department of Mathematics, China University of Mining
and Technology, Xuzhou, Jiangsu 221116, China}}}

\vspace{0.2cm} {Zhaosheng Feng\footnote{Corresponding author: zhaosheng.feng@utrgv.edu; fax: (956) 665-5091.}}\\  \vspace{6pt}
{\footnotesize {\em{ School of Mathematical and Statistical Sciences, University of Texas,  Edinburg, Texas 78539, USA}}}

\vspace{0.2cm} {Xingjie Yan\footnote{Corresponding author: yanxj04@163.com.}}\\\vspace{6pt}
{\footnotesize {\em{ Department of Mathematics, China University of Mining
and Technology, Xuzhou, Jiangsu 221116, China
}}}

\end{center}

\begin{abstract}
In this paper, we study the existence of multiple positive solutions for a class of fractional Schr\"{o}dinger-Poisson systems involving sign-changing potential and critical nonlinearities on an unbounded domain.  With the help of Nehari manifold and Ljusternik-Schnirelmann category, we investigate how the coefficient $g(x)$ of the critical nonlinearity affects the number of positive solutions. Moreover, we present a novel relationship between the number of positive solutions and the category of the global maximum set of $g(x)$.


\end{abstract}

\section{\normalsize{Introduction}}\noindent

In this paper, our main interest is devoted to multiple positive solutions of the fractional Schr\"{o}dinger-Poisson-type system:
\begin{equation}\label{s1.1}
\left\{\begin{array}{ll}
(-\Delta)^\alpha u+u+l(x)\phi u= f_\lambda(x)|u|^{q-2}u+g(x)|u|^{2_\alpha^*-2}u,~ &\text{in}~\mathbb{R}^N,\\
(-\Delta)^\frac{s}{2}\phi=\gamma_sl(x)u^2,~~~~~~~~~~~~~~~~~~~~~~~~~~~~~~~~~~~~~~~~~~~~& \text{in}~\mathbb{R}^N,
\end{array}\right.
\end{equation}
where $\alpha\in(0,1),\ q\in(1,2)$ or $q\in(4,2_\alpha^*)$, $s\in(0,N)$, $2\alpha<N\leq4\alpha$, $\gamma_s=\frac{\pi^\frac{N}{2}2^s\Gamma(\frac{s}{2})}{\Gamma(\frac{N-s}{2})}$ is a constant, $2_\alpha^*=\frac{2N}{N-2\alpha}$ is the fractional critical Sobolev exponent, and $\Gamma$ is the Gamma function. The operator $(-\Delta)^\alpha$ is the so-called fractional Laplacian defined by the Fourier transform:
$$\widehat{(-\Delta)^\alpha}u(\xi)=|\xi|^{2\alpha}\widehat{u},$$
where $\widehat{u}$ denotes the Fourier transform of $u$, and $u: \mathbb{R}^N \rightarrow\mathbb{R}$ belongs to the Schwartz class $\mathcal{S}(\mathbb{R}^N)$. For the sake of using the variational methods conveniently, $(-\Delta)^\alpha$ is equivalently defined by \cite{32}:

$$(-\Delta)^{\alpha}u(x)=-\frac{C(s)}{2}\int_{\mathbb{R}^N}\frac{u(x+y)+u(x-y)-2u(x)}{|y|^{N+2\alpha}}dy,\ \,  x\in\mathbb{R}^N,$$
where $$C(s)=\left(\int_{\mathbb{R}^N}\frac{1-cos\xi_1}{|\xi|^{N+2\alpha}}d\xi \right)^{-1},\ \, \xi=(\xi_1,\xi_2,...,\xi_N).$$
Here, $l,\, f_\lambda$ and $g$ are continuous functions on $\mathbb{R}^N$. The function $f_\lambda(x)$ is defined by $f_\lambda(x)=\lambda f_++f_-$, where $\lambda>0$ is a small parameter and $f_\pm=\pm\max\{\pm f(x),\,0\}$.

The classical Schr\"{o}dinger--Poisson type system takes the form:
\begin{equation}\label{1.2}
\left\{\begin{array}{ll}
-\Delta u+V(x)u+l(x)\phi u=f(x,u),\\
-\Delta\phi=l(x)u^2,
\end{array}\right.
\end{equation}
which arises in quantum mechanics and semiconductor theory introduced by Benci and Fortunato\cite{q}, where the unknowns $u$ and $\phi$ represent the wave functions associated
with the particle and electric potential, and the functions $V(x)$ and $l(x)$ stand for an
external potential and a nonnegative density charge, respectively. The Schr\"{o}dinger equation coupled with a Poisson equation is usually used to interpret the phenomenon that a quantum particle
interacts with an electromagnetic field. For more details about the physical background of system (\ref{1.2}), we refer the readers to \cite{q,w,e,cs} and the references therein.

System (\ref{1.2}) was extensively investigated in both bounded and unbounded domains under various assumptions on potentials and nonlinearities. Some profound results on the existence of solutions of system
(\ref{1.2}) have been obtained \cite{3,RD,AA,AA1,AA2}. When system (\ref{1.2}) does not contain the electrostatic potential $\phi$, it reduces to the classical nonlinear Schr\"{o}dinger equation:
\begin{equation}\label{e1.3}
-\Delta u+V(x)u=f(x,u),~~x\in\mathbb{R}^N,
\end{equation}
where $V(x)$ represents an external potential. Solutions of equation (\ref{e1.3}) can be used to study standing wave solutions $\psi(x,t)=u(x)e^{-iE t}$ of the equation:
\begin{equation}
i\frac{\partial \psi}{\partial t}+\Delta\psi-V(x)\psi=-f(x,u),\ \, x\in\mathbb{R}^N.
\end{equation}

It is well-known that the existence and multiplicity of positive solutions of (\ref{e1.3}) is influenced by its external potentials. Since Cingolani and Lazzo \cite{8} found the relationship between the number of positive solutions  and the category of the set of minima of $V(x)$ under certain assumptions, we have seen quite many results established based on \cite{8}. For example, He-Li-Peng obtained a similar result for a Kirchhoff type problem  \cite{9}. Figueiredo and Junior \cite{10} generalized the same result to the Schr\"{o}dinger-Kirchhoff-type problems.  When the domain is bounded and $f(x,u)=h(x)|u|^{q-2}u+g(x)|u|^{2^*-2}u$ with $2^*=\frac{2N}{N-2}$ in (\ref{e1.3}),  Lin \cite{12} presented a relationship between the number of positive solutions and the category of the global maxima set of $g(x)$. Li and Wu \cite{13} derived the same result with $h(x)$ and $g(x)$ being allowed to be sign-changing. Fan \cite{14} generalized this result to a Kirchhoff problem on an unbounded domain. More relevant results on equation (\ref{e1.3}) can be seen in \cite{15,AB,BZ,11,Djidjeli} etc.   However, for equation (\ref{1.2}) with $f(x,u)=h(x)|u|^{q-2}u+g(x)|u|^{2^*-2}u$, very little has been undertaken  on the existence and multiplicity of positive solutions related to the category of the set of $g(x)$, probably because the techniques used for studying (\ref{e1.3}) can not be directly applied.

Fractional powers of Laplacian are the infinitesimal generators of L\'{e}vy stable diffusion process and arise in anomalous diffusions in plasmas, flames propagation and chemical reactions in liquids, population dynamics, and geophysical fluid dynamics, see e.g.\cite{29,30,32}. Distinguished from the classical Laplacian operator,
the usual analytical tools for classic elliptic PDEs can not be directly applied to the problem like (\ref{s1.1}) with the fractional Laplacian operator, since $(-\Delta)^\alpha$ is
a nonlocal operator. In the past decades, continuous attention has been received on the fractional Laplacian operator or more generally integro-differential operator,  see e.g.\cite{29,21,30,31}.

In the setting of the fractional Laplacian, system (\ref{s1.1}) becomes the fractional Schr\"{o}dinger-Poisson-type systems. It is a fundamental equation in fractional quantum mechanics in the study of particles on stochastic fields modeled by L$\acute{e}$vy processes \cite{33,34}. For a detailed mathematical description of fractional Schr\"{o}dinger equation, it was presented in \cite{35}.
Zhang et al \cite{36} considered the fractional nonlinear Schr\"{o}dinger-Poisson system:
\begin{equation}
\left\{\begin{array}{ll}
(-\Delta)^s u+\lambda\phi u= g(u),~ \text{in}~\mathbb{R}^3,\\
(-\Delta)^t\phi=\lambda u^2,~~~~~~~~~~~ \text{in}~\mathbb{R}^3,
\end{array}\right.
\end{equation}
where $\lambda>0$. When $g(u)$ satisfies subcritical or critical growth conditions, they proved the existence of positive solutions by using a perturbation approach as well as asymptotic behaviors of solutions as $\lambda\rightarrow0^+.$
Teng \cite{37} considered the fractional nonlinear Schr\"{o}dinger-Poisson system:
\begin{equation}
\left\{\begin{array}{ll}
(-\Delta)^s u+V(x)u+\phi u= \mu|u|^{q-1}u+|u|^{2^*_s-2}u,~ \text{in}~\mathbb{R}^3,\\
(-\Delta)^t\phi= u^2,~~~~~~~~~~~ ~~~~~~~~~~~~~~~~~~~~~~~~~~~~~~~\text{in}~\mathbb{R}^3,
\end{array}\right.
\end{equation}
where $\mu>0$, $0<q<2^*_s-1=\frac{3+2s}{3-2s}$, $s,t\in(0,1)$ and $2s+2T>3$. The existence of a nontrivial ground state
solution was proved by using variational methods.
Murcia and Siciliano \cite{19}  investigated the problem:
\begin{equation}
\left\{\begin{array}{ll}
\varepsilon^{2\alpha}(-\Delta)^\alpha u+V(x)u+\phi u= f(u),~\, \text{in}~\mathbb{R}^N,\\
\varepsilon^{\theta}(-\Delta)^\frac{s}{2}\phi=\gamma_su^2,~~~~~~~~~~~~~~~~~~~~~~ \text{in}~\mathbb{R}^N,
\end{array}\right.
\end{equation}
where $\varepsilon>0$, $\alpha\in(0,1)$, $s\in(0,N)$, $\theta\in(0,s)$, $N\in(2\alpha,2\alpha+s)$, $\gamma_s=\frac{\pi^\frac{N}{2}2^s\Gamma(\frac{s}{2})}{\Gamma(\frac{N-s}{2})}$ is a constant, $V(x)$ is a weight potential, and $f$ satisfies the classic $(A-R)$ condition:
\begin{description}
  \item[$(A-R)$]: There is $\mu\in(4,2_\alpha^*)$ such that $0<\mu F(t):=\mu\displaystyle\int_0^tf(\tau)d\tau\leq tf(t)$ for all $t>0$.
\end{description}
By using the Ljusternik-Schnirelmann theory and the Nehari manifold method, they found a relationship between the number of positive solutions and the category of the set of minima of $V(x)$.
Recently, Luo and Tang \cite{20} studied the problem with critical Sobolev exponent:
\begin{equation}
\left\{\begin{array}{ll}
\varepsilon^{2\alpha}(-\Delta)^\alpha u+V(x)u+\phi u=|u|^{2_\alpha^*-2}u+f(u),~ \text{in}~\mathbb{R}^N,\\
\varepsilon^\theta(-\Delta)^\frac{s}{2}\phi=\gamma_su^2,~~~~~~~~~~~~~~~~~~~~~~~~~~~~~~~~~~~~ \text{in}~\mathbb{R}^N,
\end{array}\right.
\end{equation}
where $\varepsilon>0$, $\alpha\in(\frac{1}{2},1)$, $N\in(2\alpha,4\alpha)$, $s\in(N-2\alpha,N)$, $\theta\in(0,s)$,  $V$ satisfies  a local condition and $f$ satisfies:
\begin{description}
  \item[$(f):$] The function $t\mapsto\frac{f(t)}{t^3}$ is non-decreasing in $(0,\infty)$.
\end{description}
By means of the penalization techniques and the Ljusternik-Schnirelmann theory, they obtained a relationship between the number of positive solutions and the category of the set of minima of $V(x)$. However, if $f(u)=|u|^{q-2}u$, the assumptions $(A-R)$ and $(f)$ are both imply that $q>4$, which is the essential point of the obtained results in \cite{19,20}.

The above results focus on the existence of solutions of fractional Schr\"{o}dinger-Poisson system with the
influence of potential $V(x)$ and the relationship between the number of positive solutions and the category of the set of minima of potential $V(x)$. However, it is not clear if the weight potential appear in nonlinear term does the same thing. So if we consider the problem (\ref{s1.1}), it is seems natural to ask whether it is possible to relate the  number of positive solutions to the category of the set of the maximum of $g(x)$. The purpose of the present paper is to give an affirmative answer to such a question for the case $1<q<2$ and $4<q<2_\alpha^*$. For the case $2\leq q\leq4$, the primary difficulty is that we can not obtain the boundedness of the $(PS)$ sequence for system (\ref{s1.1}) don't like in the cases $1<q<2$ or $4<q<2_\alpha^*$. Some researchers used the Pohozaev-Nehari manifold to overcome this difficulty and to obtain the existence of ground states for a similar problem under assumptions that $f_\lambda(x)$ and $g(x)$ are constants, see \cite{Li}. That is because if $f_\lambda(x)$ and $g(x)$ are not constants, the terms involving the gradient of  $f_\lambda(x)$ and $g(x)$ will appear in the Pohozaev-Nehari manifold and so the boundedness of the $(PS)$ sequence still can not be obtained. It seems
that a cut-off technique is needed for this case and we shall study it in a forthcoming paper.

Before stating our main results, we introduce some conditions on $l(x),\ f_\lambda(x)$ and $g(x)$:
\begin{description}
 \item[$(H_1)$] $\displaystyle\lim_{|x|\rightarrow\infty}l(x)=0$ and $l(x)\geq0$ for all $x\in \mathbb{R}^N$.
 \item[$(H_2)$]   $f_\lambda(x)\in L^{q^*}(\mathbb{R}^N)$, where $q^*=\frac{2_\alpha^*}{2_\alpha^*-q}$.
  \item[$(H_3)$] $\displaystyle\lim_{|x|\rightarrow\infty}g(x)=g_\infty\in(0,+\infty)$ and $g(x)\geq g_\infty$ for all $x\in\mathbb{R}^N$.
 \item[$(H_4)$]There exist a non-empty closed set $M=\{z\in\mathbb{R}^N; g(z)=\displaystyle\max_{x\in\mathbb{R}^N}g(x)=1\}$ and a positive number $\rho>N$ such that $g(z)-g(x)=O(|x-z|^\rho)$ as $x\rightarrow z$ and uniformly in $z\in M$.
  \item[$(H_5)$]  $f_\lambda(x)\geq0$ for $x\in M$.
\end{description}
\begin{remark}\label{r1.1}
Let $M_r=\left\{x\in \mathbb{R}^N;\, dist(x,M)<r \right\}$ for $r>0$. Then by conditions $(H_4)-(H_5)$, there exist positive constants  $C_0,\, r_0>0$ such that
\begin{equation*}
  f_\lambda(x)>0,~\text{for}~ x\in M_{r_0},\ M_{r_0}\subset\mathbb{R}^N
\end{equation*}
and
\begin{equation*}
  g(z)-g(x)\leq C_0|x-z|^\rho ~\text{for}~ x\in B_{r_0}(z)
\end{equation*}
uniformly in $z\in M$, where $B_{r_0}(z)=\{x\in \mathbb{R}^N;|x-z| <r_0\}$. Moreover, in view of conditions $(H_3)-(H_4)$, it is easy to see that $g_\infty<1$.
\end{remark}

\begin{theorem}\label{t1.1}
Assume that $1<q<2$ and conditions $(H_1)-(H_5)$ hold. Then for each $0<\delta<r_0$, there exists $\lambda_\delta>0$ such that if $\lambda\in(0,\lambda_\delta)$, the problem $(\ref{s1.1})$ has at least $cat_{M_\delta}(M)+1$ distinct positive solutions, where $cat$ means the Ljusternik-Schnirelmann category.
\end{theorem}

For the case of $\max\left\{4,\,\frac{N}{N-2\alpha}\right\}<q<2_\alpha^*$, we impose the following conditions on $l(x)$ and $f_\lambda(x)$ instead of conditions $(H_1)-(H_2)$:
\begin{description}
 \item[$(A_1)$] $l(x)\in L^\infty(\mathbb{R}^N)$ and $l(x)\geq0$ for all $x\in \mathbb{R}^N$.
 \item[$(A_2)$]   $f_\lambda(x)\equiv\lambda f(x)$ and  $\displaystyle\lim_{|x|\rightarrow\infty}f(x)=f_\infty,\ f_\infty\in(0,+\infty)$ and $f(x)\geq f_\infty$ for $x\in\mathbb{R}^N$.
\end{description}

\begin{theorem}\label{t1.2}
Assume that conditions $(H_3)-(H_4)$ and $(A_1)-(A_2)$ hold, when $\max\left\{4,\,\frac{N}{N-2\alpha} \right\}<q<2_\alpha^*$. Then for each $0<\delta<r_0$, there exists $\lambda_\delta>0$ such that if $\lambda\in(0,\lambda_\delta)$, the problem $(1.1)$ has at least $cat_{M_\delta}(M)$ distinct positive solutions.
\end{theorem}
\begin{remark}\label{r1.2}
 It is easy to see that conditions $(A_1)-(A_2)$ are weaker than conditions $(H_1)-(H_2)$. Thus if we replace conditions $(H_1)-(H_2)$  with $(A_1)-(A_2)$,  the conclusion of this Theorem is certainly true.
\end{remark}

The innovations to prove theorems \ref{t1.1} and \ref{t1.2} are summarized as follows:
\begin{enumerate}
 \item[(1)] As in \cite{8,9,10,11,12,13,14,15}, the Nehari manifold and the Ljusternik-Schnirelmann theory are the key tools in the proof of Theorem \ref{t1.1} and Theorem \ref{t1.2}. For the proof of Theorem \ref{s1.1}, one of the major difficulties is to find the energy of the associated energy functional constrained on the Nehari manifold, this is because the variational methods used in the literature, see \cite{8,9,10,11,12,13,14,15,17,19,20}, cannot be used in the case of the growth order of $1<q<2$ .
     To overcome this obstacle,  we introduce a new version of Young's inequality for weak $L^{r,w}(\mathbb{R}^N)$ spaces (see Lemma \ref{l3.3}) and give an accurate estimate for the poisson term in the associated energy functional.

\item[(2)] Another  difficulty for proving Theorem \ref{t1.1} lies in the fact that there is no concentration--compactness result on the fractional Schr\"{o}dinger equations. In this study we introduce a scale operation (see Lemma \ref{l4.1}) and establish a global compactness result (see Lemma \ref{l4.2}).

\item[(3)]Since we use conditions $(A_1)-(A_2)$ instead of $(H_1)-(H_2)$, the resultant difficulty in the proof of Theorem 1.2 becomes harder to obtain the compactness of the $(PS)$ sequence for the associated energy functional. To this end, we propose a new analytical technique to derive a precise estimate for the $(PS)$ sequence (see Lemma \ref{l5.3}).

\end{enumerate}

The remainder of the paper is organized as follows. In Section 2, we recall some related preliminary results on fractional Sobolev spaces, and present the variational setting of the problem and properties of the corresponding Nehari manifold. In Section 3, we introduce the weak $L^{r,w}(\mathbb{R}^N)$ spaces and derive useful estimates  which play a crucial role in the proof of Theorem \ref{t1.1}. In Section 4, we construct the barycenter map and prove  Theorem \ref{t1.1} by means of the Ljusternik-Schnirelmann category theory. Section 5 is dedicated to the proof of Theorem \ref{t1.2}.


\section{\normalsize{Preliminary Results}}\noindent

In this section, we present some related preliminary results on fractional Sobolev spaces \cite{32}.
For the sake of simplicity, we denote by $\rightarrow$ (resp. $\rightharpoonup$) the strong (resp. weak) convergence. We will use $C,\,C_0,\,C_1,\,C_2,\ldots$ to denote various positive constants, and use  $L^r(\mathbb{R}^N),\ 1<r<\infty$ to represent the usual Lebesgue space of vector-valued functions with the usual $L^r$ norm in $\mathbb{R}^N$ denoted by $|\cdot|_r$.

The fractional Sobolev space $H^\alpha(\mathbb{R}^N)$ is defined as the completion of $C_c^\infty(\mathbb{R}^N)$ under the norm:
\begin{equation*}
  \|u\|_{H^\alpha}:=\left(\displaystyle\int_{\mathbb{R}^N}|(-\Delta)^\frac{\alpha}{2}u|^2dx+\displaystyle\int_{\mathbb{R}^N}|u|^2dx\right)^{\frac{1}{2}}.
\end{equation*}
It is well-known that $H^\alpha(\mathbb{R}^N)\hookrightarrow L^r(\mathbb{R}^N)$, $r\in[2,2_\alpha^*]$ and the embedding is continuous. Moreover, the embedding is compact if $\Omega\subset\mathbb{R}^N$ is bounded and $r\neq 2_\alpha^*$. Set $\dot{H}^\frac{s}{2}(\mathbb{R}^N)$ as the completion of $C_0^\infty(\mathbb{R}^N)$ under the norm:
\begin{equation*}
  \|u\|_{\dot{H}^\frac{s}{2}}:=\left(\displaystyle\int_{\mathbb{R}^N}|(-\Delta)^\frac{s}{4}u|^2dx\right)^{\frac{1}{2}}.
\end{equation*}
It is easy to see that $\dot{H}^\frac{s}{2}\hookrightarrow L^{2_\frac{s}{2}^*}(\mathbb{R}^N)$, where $2_\frac{s}{2}^*=\frac{2N}{N-s}$.

To consider the second equation in system (1.1), we fix a function $u\in H^\alpha(\mathbb{R}^N)$ and start with the problem:
\begin{equation}\label{e2.1}
\left\{\begin{array}{ll}
(-\Delta)^\frac{s}{2}\phi=\gamma_sl(x)u^2,~\text{in}~\mathbb{R}^N,\\
\phi\in \dot{H}^\frac{s}{2}(\mathbb{R}^N),
\end{array}\right.
\end{equation}
whose weak solution is a solution $\phi\in \dot{H}^\frac{s}{2}(\mathbb{R}^N)$ satisfying
\begin{equation*}
\displaystyle\int_{\mathbb{R}^N}(-\Delta)^\frac{s}{4}\phi(-\Delta)^\frac{s}{4}vdx
=\gamma_s\displaystyle\int_{\mathbb{R}^N}l(x)u^2vdx
\end{equation*}
for any $v\in \dot{H}^\frac{s}{2}(\mathbb{R}^N)$.

Note that
\begin{equation}
 \left|\displaystyle\int_{\mathbb{R}^N}l(x)u^2vdx\right|\leq|u|_\frac{4N}{N+s}^2|v|_{2^*_{\frac{s}{2}}}\leq C\|u\|_{H^\alpha}^2\|v\|_{\dot{H}^\frac{s}{2}}.
\end{equation}
So the map
\begin{equation*}
L_u:v\in \dot{H}^\frac{s}{2}(\mathbb{R}^N)\mapsto\displaystyle\displaystyle\int_{\mathbb{R}^N}l(x)u^2vdx
\end{equation*}
is linear and continuous. It follows \cite{21} that there exists a unique solution $\phi_u\in \dot{H}^\frac{s}{2}(\mathbb{R}^N)$ to system (\ref{e2.1}) and
\begin{equation}
 \phi_u=\frac{1}{|\cdot|^{N-s}}*u^2,
\end{equation}
where $*$ means convolution. Moreover, we have
\begin{equation}\label{e2.4}
 \|\phi_u\|_{\dot{H}^\frac{s}{2}}=\|L_u\|_{\mathcal{L}(\dot{H}^\frac{s}{2}(\mathbb{R}^N);\mathbb{R})}\leq C|u|_\frac{4N}{N+s}^2\leq C\|u\|_{H^\alpha}^2,
\end{equation}
where $\mathcal{L}(\dot{H}^\frac{s}{2}(\mathbb{R}^N);\mathbb{R})$ represents the set of bounded linear operators from $\dot{H}^\frac{s}{2}(\mathbb{R}^N)$ to $\mathbb{R}$.
If we define the map:
 \begin{equation*}
 A: u\in H^\alpha(\mathbb{R}^N)\mapsto\displaystyle\int_{\mathbb{R}^N}l(x)\phi_uu^2dx,
\end{equation*}
then we have
\begin{equation}\label{e2.5}
|A(u)|\leq C|\phi_u|_{2^*_{\frac{s}{2}}}|u|_\frac{4N}{N+s}^2\leq C\|\phi_u\|_{\dot{H}^\frac{s}{2}}|u|_\frac{4N}{N+s}^2\leq C|u|_\frac{4N}{N+s}^4\leq C\|u\|_{H^\alpha}^4.
\end{equation}

Some of the related properties of  $\phi_u$ and $A$ are listed below.

\begin{lemma} \cite{19}\label{l2.1} The following statements are true.
\begin{enumerate}
  \item[\textup{(i)}] For each $u\in H^\alpha(\mathbb{R}^N)$, $\phi_u\geq0$.
\item[\textup{(ii)}]   For each $u\in H^\alpha(\mathbb{R}^N)$ and $t\in \mathbb{R}$, $\phi_{tu}=t^2\phi_u$.
 \item[\textup{(iii)}] If $u_n\rightharpoonup u$ in $H^\alpha(\mathbb{R}^N)$, then $\phi_{u_n}\rightharpoonup\phi_u$ in $\dot{H}^\frac{s}{2}(\mathbb{R}^N)$.
 \item[\textup{(iv)}] If $A$ is of class $\mathcal{C}^2$ and for every $u,v\in H^\alpha(\mathbb{R}^N)$, then $A'(u)v=4\displaystyle\int_\Omega\phi_uuvdx$.
  \item[\textup{(v)}] Assume that $u_n\rightarrow u$ in $L^r(\mathbb{R}^N)$ with $1<r<2_\alpha^*$. Then $A(u_n)\rightarrow A(u)$.
   \item[\textup{(vi)}] If $u_n\rightharpoonup u$ in $H^\alpha(\mathbb{R}^N)$, then $A(u_n-u)=A(u_n)-A(u)+o(1)$.
\end{enumerate}
\end{lemma}

Substituting (2.3) into $(1.1)$ yields an equivalent equation:
\begin{equation}\label{e2.6}
(-\Delta)^\alpha u+u+l(x)\phi_u u= f_\lambda(x)|u|^{q-2}u+g(x)|u|^{2_\alpha^*-2}u,~ \text{in}~\mathbb{R}^N.
\end{equation}
To treat the nonlocal problem (\ref{e2.6}), with the help of the harmonic extension method \cite{21}, we consider a corresponding extension problem, also a local problem,  that allows us to use the classical nonlinear variational methods.

Given a function $u\in \dot{H}^\alpha(\mathbb{R}^N)$, the solution $\omega\in \dot{X}^\alpha(\mathbb{R}_+^{N+1})$ of
\begin{displaymath}
\left\{\begin{array}{ll}
div(y^{1-2\alpha}\nabla\omega)=0,~ &\text{in}~\mathbb{R}_+^{N+1},\\
\omega=u,~~&\text{on}~\mathbb{R}^N\times\{0\},
\end{array}\right.
\end{displaymath}
is called $\alpha$-harmonic extension $\omega=E_\alpha(u)$ of $u$, where $\dot{X}^\alpha(\mathbb{R}_+^{N+1})$ is defined as the completion of $C_0^\infty(\mathbb{R}_+^{N+1})$ under the norm:
\begin{equation*}
 \|\omega\|_{\dot{X}^\alpha}=\left(\displaystyle\int_{\mathbb{R}_+^{N+1}}k_\alpha y^{1-2\alpha}|\nabla\omega|^2dxdy\right)^{\frac{1}{2}}
\end{equation*}
and $k_\alpha$ is a normalization constant. We know that $E_\alpha(\cdot)$ is an isometry between $\dot{H}^\alpha(\mathbb{R}^N)$ and $\dot{X}^\alpha(\mathbb{R}_+^{N+1})$, and
\begin{equation*}
 \|\omega\|_{\dot{X}^\alpha}=\sqrt{k_\alpha}\|u\|_{\dot{H}^\alpha},~~ u\in \dot{X}^\alpha(\mathbb{R}^N).
\end{equation*}

We can further define $X^\alpha(\mathbb{R}_+^{N+1})$ by the completion of $C_0^\infty(\mathbb{R}_+^{N+1})$ under the norm:
\begin{equation*}
 \|\omega\|_{X^\alpha}=\left(\displaystyle\int_{\mathbb{R}_+^{N+1}}k_\alpha y^{1-2\alpha}|\nabla\omega|^2dxdy+\displaystyle\int_{\mathbb{R}^N}k_\alpha|\omega(x,0)|^2dx\right)^{\frac{1}{2}}.
\end{equation*}
It is easy to see that $E_\alpha(\cdot)$ is an isometry between $H^\alpha(\mathbb{R}^N)$ and $X^\alpha(\mathbb{R}_+^{N+1})$.

Re-formulate (\ref{e2.6}) as follows:
\begin{equation}\label{e2.7}
\left\{\begin{array}{ll}
div(y^{1-2\alpha}\nabla\omega)=0,~ &\text{in}~\mathbb{R}_+^{N+1},\\
-k_\alpha\frac{\partial\omega}{\partial\nu}=-\omega-\phi_\omega\omega+f_\lambda(x)|\omega|^{q-2}\omega+g(x)|\omega|^{2_\alpha^*-2}\omega,~~&\text{in}~\mathbb{R}^N\times\{0\},
\end{array}\right.
\end{equation}
where
\begin{equation*}
  -k_\alpha\frac{\partial\omega}{\partial\nu}=-k_\alpha\displaystyle\lim_{y\rightarrow0^+}y^{1-2\alpha}\frac{\partial\omega}{\partial y}(x,y)= (-\Delta)^\alpha u(x).
\end{equation*}
For simplicity, we set $k_\alpha=1$.
If $\omega$ is a weak solution of system (\ref{e2.7}), then  $u=tr(\omega)=\omega(x,0)$. The trace of $\omega$ is a weak solution of equation (\ref{e2.6}). The converse is also true.

Now we introduce some basic results on the spaces $X^\alpha(\mathbb{R}_+^{N+1})$ and $L^r(\mathbb{R}^N)$.
\begin{lemma} \cite{31}\label{l2.2}
The embedding $X^\alpha(\mathbb{R}_+^{N+1})\hookrightarrow L^r(\mathbb{R}^N)$ is continuous for $r\in[2,2_\alpha^*]$ and locally compact for any $r\in [2,2_\alpha^*)$.
\end{lemma}

\begin{lemma} \cite{31}\label{l2.3}
For each $\omega\in X^\alpha(\mathbb{R}_+^{N+1})$, there holds
\begin{equation}
S\left(\displaystyle\int_{\mathbb{R}^N}|u(x)|^\frac{2N}{N-2\alpha}dx\right)^\frac{N-2\alpha}{N}\leq\displaystyle\int_{\mathbb{R}_+^{N+1}}y^{1-2\alpha}|\nabla\omega|^2dxdy,
\end{equation}
where $u=tr(\omega)$. The best constant is given by
 \begin{equation*}
  S=\frac{2\pi^\alpha\Gamma(\frac{2-2\alpha}{2})\Gamma(\frac{N+2\alpha}{2})(\Gamma(\frac{N}{2}))^\frac{2\alpha}{N}}{\Gamma(\alpha)\Gamma(\frac{N-2\alpha}{2})(\Gamma(N))^\frac{2\alpha}{N}},
\end{equation*}
and it is attained when $u=\omega(x,0)$ takes the form
\begin{equation}\label{e2.9}
u_\varepsilon(x)=\frac{C\varepsilon^\frac{N-2\alpha}{2}}{(\varepsilon^2+|x|^2)^\frac{N-2\alpha}{2}}
\end{equation}
for $\varepsilon>0$, and $\omega_\varepsilon=E_\alpha(u_\varepsilon)$ and
$\|\omega_\varepsilon\|_{\dot{X}^\alpha}^2=\displaystyle\int_{\mathbb{R}^N}|\omega_\varepsilon(x,0)|^\frac{2N}{N-2\alpha}dx=S^\frac{N}{2\alpha}$.
\end{lemma}

\begin{lemma} \cite{20}\label{l2.4}
The following two statements are true:
\begin{enumerate}
  \item[\textup{(i)}] Let $R>0$ and $\mathrm{T}$ be a subset of $X^\alpha(\mathbb{R}_+^{N+1})$ such that
   \begin{equation*}
 \displaystyle\sup_{\omega\in\mathrm{T}}\displaystyle\int_{\mathbb{R}_+^{N+1}}y^{1-2\alpha}|\nabla\omega|^2dxdy<\infty.
\end{equation*}
Then $\mathrm{T}$ is pre-compact in $L^2(B_R^+,y^{1-2\alpha})$, where $B_R^+:=\{(x,y)\in\mathbb{R}^N\times\mathbb{R}_+;|(x,y)|<R\}$ and the weight Lebesgue space $L^2(B_R^+,y^{1-2\alpha})$ is equipped with the norm:
\begin{equation*}
 \|\omega\|_{L^2(B_R^+,y^{1-2\alpha})}=\left(\displaystyle\int_{B_R^+}y^{1-2\alpha}|\omega|^2dxdy\right)^{\frac{1}{2}}.
\end{equation*}
  \item[\textup{(ii)}]For  $\omega\in X^\alpha(\mathbb{R}_+^{N+1})$, there holds
  \begin{equation*}
  \left(\displaystyle\int_{\mathbb{R}_+^{N+1}}y^{1-2\alpha}|\omega|^{2\gamma}dxdy\right)^{\frac{1}{2\gamma}}\leq C_0\left(\displaystyle\int_{\mathbb{R}_+^{N+1}}y^{1-2\alpha}|\nabla\omega|^2dxdy\right)^{\frac{1}{2}},
  \end{equation*}
  where $\gamma=1+\frac{2}{N-2\alpha}$ and $C_0$ is a positive constant independent of the choice of $\omega\in X^\alpha(\mathbb{R}_+^{N+1})$.
\end{enumerate}
\end{lemma}

Define the energy functional associated with system (\ref{e2.7}) by
 \begin{equation*}
 \begin{split}
 I_\lambda(\omega)=&\frac{1}{2}\|\omega\|_{X^\alpha}^2+\frac{1}{4}\displaystyle\int_{\mathbb{R}^N}l(x)\phi_{\omega(x,0)}|\omega(x,0)|^2dx-\frac{1}{q}\displaystyle\int_{\mathbb{R}^N}f_\lambda(x)|\omega(x,0)|^qdx\\
 &-\frac{1}{2_\alpha^*}\displaystyle\int_{\mathbb{R}^N}g(x)|\omega(x,0)|^{2_\alpha^*}dx,
\end{split}
\end{equation*}
where $\omega\in X^\alpha(\mathbb{R}_+^{N+1})$.
On account of $I_\lambda$ being not bounded from below on $X^\alpha(\mathbb{R}_+^{N+1})$, we consider the behaviors of $I_\lambda$ on the Nehari manifold:
 \begin{equation*}
 N_\lambda:=\{\omega\in X^\alpha(\mathbb{R}_+^{N+1})\backslash\{0\}; I'_\lambda(\omega)\omega=0\}.
\end{equation*}
 Then $\omega\in N_\lambda$ if and only if
 \begin{equation}\label{e2.10}
 \|\omega\|_{X^\alpha}^2+\displaystyle\int_{\mathbb{R}^N}l(x)\phi_{\omega(x,0)}|\omega(x,0)|^2dx-\displaystyle\int_{\mathbb{R}^N}f_\lambda(x)|\omega(x,0)|^qdx-\displaystyle\int_{\mathbb{R}^N}g(x)|\omega(x,0)|^{2_\alpha^*}dx=0.
\end{equation}

Set
 \begin{equation}\label{ee2.12}
 \begin{split}
 \Psi_\lambda(\omega):=&\|\omega\|_{X^\alpha}^2+\displaystyle\int_{\mathbb{R}^N}l(x)\phi_{\omega(x,0)}|\omega(x,0)|^2dx-\displaystyle\int_{\mathbb{R}^N}f_\lambda(x)|\omega(x,0)|^qdx\\&
 -\displaystyle\int_{\mathbb{R}^N}g(x)|\omega(x,0)|^{2_\alpha^*}dx.
\end{split}
\end{equation}
Then for $\omega\in N_\lambda$ we have
\begin{align}\label{e2.12}
&\Psi'_\lambda(\omega)\omega\notag
\nonumber\\=& 2\|\omega\|_{X^\alpha}^2+4\displaystyle\int_{\mathbb{R}^N}l(x)\phi_{\omega(x,0)}|\omega(x,0)|^2dx-q \displaystyle\int_{\mathbb{R}^N}f_\lambda(x)|\omega(x,0)|^qdx\nonumber\\&~~~~-2_\alpha^*\displaystyle\int_{\mathbb{R}^N}g(x)|\omega(x,0)|^{2_\alpha^*}dx\notag
\nonumber\\=& -2\|\omega\|_{X^\alpha}^2+(4-q) \displaystyle\int_{\mathbb{R}^N}f_\lambda(x)|\omega(x,0)|^qdx-(2_\alpha^*-4)\displaystyle\int_{\mathbb{R}^N}g(x)|\omega(x,0)|^{2_\alpha^*}dx
\nonumber\\=& (2-q)\|\omega\|_{X^\alpha}^2+(4-q)\displaystyle\int_{\mathbb{R}^N}l(x)\phi_{\omega(x,0)}|\omega(x,0)|^2dx+(q-2_\alpha^*)\displaystyle\int_{\mathbb{R}^N}g(x)|\omega(x,0)|^{2_\alpha^*}dx.
\end{align}

We split $N_\lambda$ into the following three parts:
\begin{align}
N_\lambda^+&=\{\omega\in N_\lambda;\Psi'_\lambda(\omega)\omega>0\},\notag\\
N_\lambda^0&=\{u\in N_\lambda;\Psi'_\lambda(\omega)\omega=0\},\notag\\
N_\lambda^-&=\{u\in N_\lambda;\Psi'_\lambda(\omega)\omega<0\}.\notag
\end{align}

\begin{lemma}\label{l2.5}
$I_\lambda$ is coercive and bounded from below on $N_\lambda$.
\end{lemma}
\begin{proof}
For $\omega\in N_\lambda$, by H\"{o}lder's and Young's inequalities, it follows  Lemma \ref{l2.2} and (\ref{e2.10}) that
\begin{align}\label{e2.13}
I_\lambda(\omega)&=I_\lambda(\omega)-\frac{1}{4}\Psi'_\lambda(\omega)\omega\notag\\
&=\frac{1}{4}\|\omega\|_{X^\alpha}^2+\left(\frac{1}{4}-\frac{1}{2_\alpha^*}\right)\displaystyle\int_{\mathbb{R}^N}g(x)|\omega(x,0)|^{2_\alpha^*}dx-\left(\frac{1}{q}-\frac{1}{4}\right)\displaystyle\int_{\mathbb{R}^N}f_\lambda(x)|\omega(x,0)|^qdx\notag\\
&\geq\frac{1}{4}\|\omega\|_{X^\alpha}^2-\lambda\left(\frac{1}{q}-\frac{1}{4}\right)C|f_+|_{q^*}S^{-\frac{q}{2}}\|\omega\|_{X^\alpha}^q\nonumber\\
&\geq\frac{1}{4}\|\omega\|_{X^\alpha}^2-\frac{1}{4}\|\omega\|_{X^\alpha}^2-C\lambda^\frac{2}{2-q}\notag\\
&=-C\lambda^\frac{2}{2-q},
\end{align}
where $C$ is a positive constant independent of the choice of $\omega\in X^\alpha(\mathbb{R}_+^{N+1})$ and $\lambda>0$.
\end{proof}

\begin{lemma}\label{l2.6}
Suppose that $\omega_0$ is a local minimizer for $I_\lambda$ on $N_\lambda$ and $\omega_0 \not\in N_\lambda^0.$ Then we have $I_\lambda'(\omega_0)=0$.
\end{lemma}

The proof is closely similar to \cite[Lemma 4.3]{22Lawson} and we omit it.

\begin{lemma}\label{l2.7}
There exists $\Lambda_1>0$ such that $N_\lambda^0=\emptyset$ for $\lambda\in(0,\Lambda_1)$.
\end{lemma}
\begin{proof}
 Suppose that  $\omega\in N_\lambda^0$, it follows (\ref{e2.12})-(\ref{e2.13}) and Lemma \ref{l2.2} that
  \begin{equation*}
2\|\omega\|_{X^\alpha}^2\leq\lambda\displaystyle\int_{\mathbb{R}^N} f_\lambda(x)|\omega(x,0)|^qdx\leq\lambda |f_+|_{q^*}S^{-\frac{q}{2}}\|\omega\|_{X^\alpha}^q
 \end{equation*}
and
\begin{equation*}
(2-q)\|\omega\|_{X^\alpha}^2\leq (2_\alpha^*-q)\displaystyle\int_{\mathbb{R}^N}g(x)|\omega(x,0)|^{2_\alpha^*}dx\leq (2_\alpha^*-q)S^{-\frac{2_\alpha^*}{2}}\|\omega\|_{X^\alpha}^{2_\alpha^*}.
 \end{equation*}
So we have
\begin{equation*}
C_1\leq\|\omega\|_{X^\alpha}\leq \lambda^{\frac{1}{2-q}}C_2,
 \end{equation*}
where $C_1,\, C_2>0$ are independent of the choice of $\omega\in X^\alpha(\mathbb{R}_+^{N+1})$ and $\lambda>0$. When $\lambda$ is sufficiently small, it yields a contradiction.
\end{proof}

By Lemma \ref{l2.7}, we get $N_\lambda=N_\lambda^+\cup N_\lambda^-$. Set
 \begin{equation*}
\alpha_\lambda^+=\displaystyle\inf_{\omega\in N_\lambda^+}I_\lambda(\omega)~\text{and}~\alpha_\lambda^-=\displaystyle\inf_{\omega\in N_\lambda^-}I_\lambda(\omega).
 \end{equation*}

\begin{lemma}\label{l2.8}
The following two statements are true.
\begin{enumerate}
  \item[\textup{(i)}] $\alpha_\lambda^+<0$.
  \item[\textup{(ii)}] There exists $\Lambda_2\in(0,\Lambda_1)$ such that $\alpha_\lambda^->d_0$ for some $d_0>0$ and $\lambda\in (0,\Lambda_2).$
\end{enumerate}
In particular, $\alpha_\lambda^+=\displaystyle\inf_{u\in N_\lambda}I_\lambda(u)$ for $\lambda\in (0,\Lambda_2).$
\end{lemma}
\begin{proof}
(i). For any $\omega\in N_\lambda^+$, it follows (\ref{e2.12}) that
  \begin{equation*}
(2-q)\|\omega\|_{X^\alpha}^2+(4-q)\displaystyle\int_{\mathbb{R}^N} l(x)\phi_{\omega(x,0)}|\omega(x,0)|^2dx>(2_\alpha^*-q)\displaystyle\int_{\mathbb{R}^N}g(x)|\omega(x,0)|^{2_\alpha^*}dx.
 \end{equation*}
Then we have
\begin{align}
\begin{split}
 I_\lambda(\omega)&=I_\lambda(\omega)-\frac{1}{q}I'_\lambda(\omega)\omega\notag
\\&=\left(\frac{1}{2}-\frac{1}{q}\right)\|\omega\|_{X^\alpha}^2+\left(\frac{1}{4}-\frac{1}{q}\right)\displaystyle\int_{\mathbb{R}^N}l(x)\phi_{\omega(x,0)} |\omega(x,0)|^2dx\\&~~~~~+\left(\frac{1}{q}-\frac{1}{2_\alpha^*}\right)\displaystyle\int_{\mathbb{R}^N}g(x)|\omega(x,0)|^{2_\alpha^*}dx\notag
\\&<\frac{q-2}{4q}\|\omega\|_{X^\alpha}^2+\frac{q-4}{4q}\displaystyle\int_{\mathbb{R}^N}l(x)\phi_{\omega(x,0)} |\omega(x,0)|^2dx+\frac{2_\alpha^*-q}{2_\alpha^*q}\displaystyle\int_{\mathbb{R}^N} g(x)|\omega(x,0)|^6dx\notag
\\&<-\frac{(2_\alpha^*-4)(2_\alpha^*-q)}{4q2_\alpha^*}\displaystyle\int_{\mathbb{R}^N}g(x)|\omega(x,0)|^{2_\alpha^*}dx\nonumber\\&<0,\notag
\end{split}
\end{align}
and thus
\begin{equation*}
\alpha_\lambda^+<0.
 \end{equation*}
(ii). For any $\omega\in N_\lambda^-$, it follows (\ref{e2.12}) that
\begin{equation*}
(2-q)\|\omega\|_{X^\alpha}^2<(2_\alpha^*-q)\displaystyle\int_{\mathbb{R}^N}g(x)|\omega(x,0)|^{2_\alpha^*}dx\leq (2_\alpha^*-q)S^{-\frac{2_\alpha^*}{2}}\|\omega\|_{X^\alpha}^{2_\alpha^*}
 \end{equation*}
and
\begin{equation}\label{e2.14}
\|\omega\|_{X_0^\alpha}\geq\left(\frac{2-q}{2_\alpha^*-q}S^\frac{2_\alpha^*}{2}\right)^\frac{1}{2_\alpha^*-2}.
 \end{equation}
Using (\ref{e2.13}) and (\ref{e2.14}), we  have
\begin{equation*}
I_\lambda(\omega)\geq\|\omega\|_{X^\alpha}^q\left(\frac{1}{4}\|\omega\|_{X^\alpha}^{2-q}-\lambda\left(\frac{1}{q}-\frac{1}{4}\right)|f_+|_{q^*}S^{-\frac{q}{2}}\right)\geq d_0>0,
 \end{equation*}
 for small  $\lambda>0$  and some constant $d_0>0$ independent of the choice of $\omega\in N_\lambda^-$.
\end{proof}

For each $\omega\in X^\alpha(\mathbb{R}_+^{N+1})\backslash\{0\}$, we set
\begin{equation*}
  s(t)=t^{2-q}\|\omega\|_{X^\alpha}^2+t^{4-q}\displaystyle\int_{\mathbb{R}^N}l(x)\phi_{\omega(x,0)} |\omega(x,0)|^2dx-t^{2_\alpha^*-q}\displaystyle\int_{\mathbb{R}^N} g(x)|\omega(x,0)|^{2_\alpha^*}dx,
\end{equation*}
for $t\geq0$. Then $s(0)=0$, $s(t)>0$ for small $t>0$ small,  and $s(t)\rightarrow-\infty$ as $t\rightarrow\infty$. There exists $t_{\max}>0$ such that
\begin{equation*}
s(t_{\max})=\sup_{t\geq0}s(t)>0.
\end{equation*}
It is also easy to see that $s(t)$ is increasing on $(0,t_{\max})$ and decreasing on $(t_{\max},\infty)$.

\begin{lemma}\label{l2.9}
For each $\omega\in X^\alpha(\mathbb{R}_+^{N+1})\backslash\{0\}$, there exists $\Lambda_3\in (0,\Lambda_2)$ such that if $\lambda\in(0,\Lambda_3)$, the following two statements are true.
\begin{enumerate}
  \item[\textup{(i)}] If $\displaystyle\int_{\mathbb{R}^N} f_\lambda(x)|\omega(x,0)|^qdx\leq0,$ then there is a unique $t^-=t^-(\omega)>t_{max}$ such that $t^-\omega\in N_\lambda^-$ and $I_\lambda(t\omega)$ is increasing on $(0,t^-)$ and decreasing on $(t^-,\infty)$. Moreover, $I_\lambda(t^-\omega)=\displaystyle\sup_{t\geq0}I_\lambda(t\omega)$.
  \item[\textup{(ii)}] If $\displaystyle\int_{\mathbb{R}^N} f_\lambda(x)|\omega(x,0)|^qdx>0,$ then there is a unique $0<t^+=t^+(\omega)<t_{max}<t^-$ such that $t^-\omega\in N_\lambda^-,\ t^+\omega\in N_\lambda^+,\  I_\lambda(t\omega)$ is decreasing on $(0,t^+)$, increasing on $(t^+,t^-)$ and decreasing on $(t^-,\infty)$. Moreover, $I_\lambda(t^+\omega)=\displaystyle\inf_{0\leq t\leq t_{max}}I_\lambda(t\omega); I_\lambda(t^-\omega)=\displaystyle\sup_{t\geq t^+}I_\lambda(t\omega)$.
\end{enumerate}
\end{lemma}
\begin{proof}
For each $\omega\in X^\alpha(\mathbb{R}_+^{N+1})\backslash\{0\}$,  we get
\begin{align}\label{e2.15}
s(t_{max})&\geq\displaystyle\max_{t\geq0}\left(t^{2-q}\|\omega\|_{X^\alpha}^2-t^{2_\alpha^*-q}\displaystyle\int_{\mathbb{R}^N} g(x)|\omega(x,0)|^{2_\alpha^*}dx\right)\notag
\\&=\left(\frac{(2-q)\|\omega\|_{X^\alpha}^2}{(2_\alpha^*-q)\displaystyle\int_{\mathbb{R}^N} g(x)|\omega(x,0)|^{2_\alpha^*}dx}\right)^{\frac{2-q}{2_\alpha^*-2}}\|\omega\|_{X^\alpha}^2\notag
\\&~~~~-\left(\frac{(2-q)\|\omega\|_{X^\alpha}^2}{(2_\alpha^*-q)\displaystyle\int_{\mathbb{R}^N} g(x)|\omega(x,0)|^{2_\alpha^*}dx}\right)^{\frac{2_\alpha^*-q}{2_\alpha^*-2}}\displaystyle\int_{\mathbb{R}^N} g(x)|\omega(x,0)|^{2_\alpha^*}dx\notag
\\&=\|\omega\|_{X^\alpha}^q\left[\left(\frac{2-q}{2_\alpha^*-q}\right)^{\frac{2-q}{2_\alpha^*-2}}-\left(\frac{2-q}{2_\alpha^*-q}\right)^{\frac{2_\alpha^*-q}{2_\alpha^*-2}}\right]\left(\frac{\|\omega\|_{X^\alpha}^{2_\alpha^*}}{\displaystyle\int_{\mathbb{R}^N} g(x)|\omega(x,0)|^{2_\alpha^*}dx}\right)^{\frac{2-q}{2_\alpha^*-2}}\notag
\\&\geq\|\omega\|_{X^\alpha}^q\left(\frac{2_\alpha^*-2}{2_\alpha^*-q}\right)\left(\frac{2-q}{2_\alpha^*-q}\right)^{\frac{2-q}{2_\alpha^*-2}}C(S),
\end{align}
where $C(S)>0$ is a constant dependent on $S$.

We need to consider two cases.

Case 1: $\displaystyle\int_{\mathbb{R}^N} f_\lambda(x)|\omega(x,0)|^qdx\leq0.$
There is a unique $t^->t_{max}$ such that $s(t^-)=\displaystyle\int_{\mathbb{R}^N} f_\lambda(x)|\omega(x,0)|^qdx$ and $s'(t^-)<0,$ which implies
$t^-\omega\in N_\lambda^-$. Moreover, $I_\lambda(tu)$ is increasing on $(0,t^-)$ and decreasing on $(t^-,\infty)$. So we have
\begin{equation*}
  I_\lambda(t^-u)=\displaystyle\sup_{t\geq0}I_\lambda(tu).
\end{equation*}

Case 2: $\displaystyle\int_{\mathbb{R}^N} f_\lambda(x)|\omega(x,0)|^qdx>0.$ It follows  (\ref{e2.15}) that
\begin{align}
s(0)=0&<\displaystyle\int_{\mathbb{R}^N} f_\lambda(x)|\omega(x,0)|^qdx\leq\lambda |f_+|_{q^*} S^{-\frac{q}{2}}\|\omega\|_{X^\alpha}^q\notag
\\&<\|\omega\|_{X^\alpha}^q\left(\frac{2_\alpha^*-2}{2_\alpha^*-q}\right)\left(\frac{2-q}{2_\alpha^*-q}\right)^{\frac{2-q}{2_\alpha^*-2}}C(S)\nonumber\\&\leq s(t_{max})\notag
\end{align}
for small $\lambda>0$. There are a unique $t^+$ and a unique $t^-$ such that $0<t^+<t_{max}<t^-$. Then
\begin{equation*}
s(t^+)=\displaystyle\int_{\mathbb{R}^N} f_\lambda(x)|\omega(x,0)|^qdx=s(t^-)
\end{equation*}
and
\begin{equation*}
s'(t^+)>0>s'(t^-),
\end{equation*}
which implies that $t^+\omega\in N_\lambda^+,$ $t^-\omega\in N_\lambda^-,$ and $I_\lambda(t^-\omega)\geq I_\lambda(t\omega)\geq I_\lambda(t^+\omega)$
for each $t\in[t^+,t^-]$. We further can get $I_\lambda(t^+\omega)\leq I_\lambda(t\omega)$ for $t\in[0,t^+].$
In other words, $I_\lambda(t\omega)$ is decreasing on $(0,t^+)$, increasing on $(t^+,t^-)$ and decreasing on $(t^-,\infty)$. Consequently, we obtain
\begin{equation*}
 I_\lambda(t^+\omega)=\displaystyle\inf_{0\leq t\leq t_{max}}I_\lambda(t\omega)~\text{and}~I_\lambda(t^-\omega)=\displaystyle\sup_{t\geq t^+}I_\lambda(t\omega).
\end{equation*}
\end{proof}

\begin{lemma}\label{l2.10}
$I_\lambda$ satisfies the $(PS)_c$-condition for $c\in \left(-\infty,\,\alpha_\lambda^++\frac{\alpha}{N}S^\frac{N}{2\alpha} \right)$.
\end{lemma}

\begin{proof}
Let $\{\omega_n\}\subset X^\alpha(\mathbb{R}_+^{N+1})$ be a $(PS)_c$ sequence for $I_\lambda$ with $c\in \left(-\infty,\,\alpha_\lambda^++\frac{\alpha}{N}S^\frac{N}{2\alpha}\right)$. Note that
\begin{align*}
&c+1+\|\omega_n\|_{X^\alpha}\nonumber\\
\geq& I_\lambda(\omega_n)-\frac{1}{4}I'_\lambda(\omega_n)\omega_n\notag\\
=&\frac{1}{4}\|\omega_n\|_{X^\alpha}^2+\left(\frac{1}{4}-\frac{1}{2_\alpha^*}\right)\displaystyle\int_{\mathbb{R}^N}g(x)|\omega_n(x,0)|^{2_\alpha^*}dx-\left(\frac{1}{q}-\frac{1}{4}\right)\displaystyle\int_{\mathbb{R}^N}f_\lambda(x)|\omega_n(x,0)|^qdx\notag\\
\geq&\frac{1}{4}\|\omega_n\|_{X^\alpha}^2-\lambda C|f_+|_{q^*}S^{-\frac{q}{2}}\|\omega_n\|_{X^\alpha}^q,
\end{align*}
It follows that $\{\omega_n\}$ is bounded in $X^\alpha(\mathbb{R}_+^{N+1})$. So there exists a subsequence, still denoted by $\{\omega_n\}$ and $\omega\in X^\alpha(\mathbb{R}_+^{N+1})$, such that $\omega_n\rightharpoonup\omega$ in $X^\alpha(\mathbb{R}_+^{N+1})$.

We claim that
\begin{equation}\label{e2.16}
\displaystyle\int_{\mathbb{R}^N} f_\lambda(x)|\omega_n(x,0)|^qdx\rightarrow\displaystyle\int_{\mathbb{R}^N} f_\lambda(x)|\omega(x,0)|^qdx,\ \, \text{as}\ \, n\rightarrow\infty.
\end{equation}
 Indeed, for any $\varepsilon>0$, it follows $(H_2)$ that there exists $R>0$ such that
\begin{equation*}
\left(\displaystyle\int_{\mathbb{R}^N\backslash B_R}|f_\lambda(x)|^{q^*}dx\right)^\frac{1}{q^*}<\varepsilon,
\end{equation*}
 \begin{align}\label{e2.17}
\left|\displaystyle\int_{\mathbb{R}^N\backslash B_R}f_\lambda(x)|\omega_n(x,0)|^qdx\right|&\leq\left(\displaystyle\int_{\mathbb{R}^N\backslash B_R}|f_\lambda(x)|^{q^*}dx\right)^\frac{1}{q^*}\left(\displaystyle\int_{\mathbb{R}^N}|\omega_n(x,0)|^{2_\alpha^*}dx\right)^\frac{q}{2_\alpha^*}\notag
\\&\leq C\left(\displaystyle\int_{\mathbb{R}^N\backslash B_R}|f_\lambda(x)|^{q^*}dx\right)^\frac{1}{q^*}\|\omega_n\|_{X^\alpha}^q\nonumber\\&\leq C\varepsilon
\end{align}
and thus
\begin{align}\label{e2.18}
\left|\displaystyle\int_{\mathbb{R}^N\backslash B_R}f_\lambda(x)|\omega(x,0)|^qdx\right|&\leq\left(\displaystyle\int_{\mathbb{R}^N\backslash B_R}|f_\lambda(x)|^{q^*}dx\right)^\frac{1}{q^*}\left(\displaystyle\int_{\mathbb{R}^N}|\omega(x,0)|^{2_\alpha^*}dx\right)^\frac{q}{2_\alpha^*}\notag
\\&\leq C\left(\displaystyle\int_{\mathbb{R}^N\backslash B_R}|f_\lambda(x)|^{q^*}dx\right)^\frac{1}{q^*}\|\omega\|_{X^\alpha}^q\nonumber\\&\leq C\varepsilon.
\end{align}

It follows H\"{o}lder's inequality and Lemma \ref{l2.2} that
\begin{align}\label{e2.19}
&\left|\displaystyle\int_{B_R} f_\lambda(x)|\omega_n(x,0)|^qdx-\displaystyle\int_{B_R} f_\lambda(x)|\omega(x,0)|^qdx\right|\notag
\\&\leq|f_\lambda|_\infty\displaystyle\int_{B_R} |\omega_n(x,0)-\omega(x,0)|^qdx\notag
\\&\leq C|f_\lambda|_\infty\left(\displaystyle\int_{B_R}|\omega_n(x,0)-\omega(x,0)|^rdx\right)^\frac{q}{r}\rightarrow0,\ \, \text{as}\ \, n\rightarrow\infty.
\end{align}
where $2<r<2_\alpha^*$ and $C>0$.

In view of (\ref{e2.17})-(\ref{e2.19}), we arrive at (\ref{e2.16}).

We now prove that
\begin{equation}\label{e2.200}
\displaystyle\int_{\mathbb{R}^N} l(x)\phi_{\omega_n(x,0)}|\omega_n(x,0)|^2dx\rightarrow\displaystyle\int_{\mathbb{R}^N} l(x)\phi_{\omega(x,0)}|\omega(x,0)|^2dx, \ \, \text{as}\ \, n\rightarrow\infty.
\end{equation}
For any $\varepsilon>0$, it follows $(H_1)$ that there exists $R>0$ such that
\begin{equation*}
|l(x)|<\varepsilon,~~\text{for}~|x|\geq R.
\end{equation*}
From (\ref{e2.5}), we deduce that
\begin{align}\label{e2.21}
\displaystyle\int_{\mathbb{R}^N\backslash B_R} l(x)\phi_{\omega_n(x,0)}|\omega_n(x,0)|^2dx&\leq\varepsilon\displaystyle\int_{\mathbb{R}^N} \phi_{\omega_n(x,0)}|\omega_n(x,0)|^2dx\notag
\\&\leq C\varepsilon|\omega_n(x,0)|^4_\frac{4N}{N+s}\nonumber\\&\leq C\varepsilon\|\omega_n\|_{X^\alpha}^4\nonumber\\&\leq C\varepsilon
\end{align}
and
\begin{align}\label{e2.22}
\displaystyle\int_{\mathbb{R}^N\backslash B_R} l(x)\phi_{\omega(x,0)}|\omega(x,0)|^2dx&\leq\varepsilon\displaystyle\int_{\mathbb{R}^N} \phi_{\omega(x,0)}|\omega(x,0)|^2dx\notag
\\&\leq C\varepsilon|\omega(x,0)|^4_\frac{4N}{N+s}\nonumber\\&\leq C\varepsilon\|\omega\|_{X^\alpha}^4\nonumber\\&\leq C\varepsilon.
\end{align}

By H\"{o}lder's inequality and (\ref{e2.5}), it follows Lemmas \ref{l2.1} (vi) and \ref{l2.2} that
\begin{align}\label{e2.23}
&\left|\displaystyle\int_{B_R} l(x)\phi_{\omega_n(x,0)}|\omega_n(x,0)|^2dx-\displaystyle\int_{B_R} l(x)\phi_{\omega(x,0)}|\omega(x,0)|^2dx\right|\notag
\\&\leq|l|_\infty\displaystyle\int_{B_R} \phi_{\omega_n(x,0)-\omega(x,0)}|\omega_n(x,0)-\omega(x,0)|^2dx\notag
\\&\leq C|\omega_n(x,0)-\omega(x,0)|^4_\frac{4N}{N+s}\rightarrow0,\ \, \text{as}\ \, n\rightarrow\infty.
\end{align}
Using (\ref{e2.21})-(\ref{e2.23}) leads to (\ref{e2.200}).

Setting $\Psi_n=\omega_n-\omega$ and applying the Br\'{e}zis-Lieb Lemma \cite{22}, we obtain
\begin{itemize}
  \item $\|\Psi_n\|_{X^\alpha}^2=\|\omega_n\|_{X^\alpha}^2-\|\omega\|_{X^\alpha}^2+o_n(1)$;
  \item $\displaystyle\int_{\mathbb{R}^N} g(x)|\Psi_n(x,0)|^{2_\alpha^*}dx=\displaystyle\int_{\mathbb{R}^N} g(x)|\omega_n(x,0)|^{2_\alpha^*}dx-\displaystyle\int_{\mathbb{R}^N} g(x)|\omega(x,0)|^{2_\alpha^*}dx+o_n(1)$.
\end{itemize}
By virtue of Lemmas \ref{l2.1} (iii) and \ref{l2.2}, it is easy to obtain that $I'_\lambda(\omega)=0$, and
\begin{equation}\label{e2.24}
\frac{1}{2}\|\Psi_n\|_{X^\alpha}^2-\frac{1}{2_\alpha^*}\displaystyle\int_{\mathbb{R}^N} g(x)|\Psi_n(x,0)|^{2_\alpha^*}dx=c-I_\lambda(\omega)+o_n(1)
\end{equation}
and
\begin{equation}\label{e2.25}
o_n(1)=I'_\lambda(\omega_n)\Psi_n=(I'_\lambda(\omega_n)-I'_\lambda(\omega))\Psi_n=\|\Psi_n\|_{X^\alpha}^2-\displaystyle\int_{\mathbb{R}^N} g(x)|\Psi_n(x,0)|^{2_\alpha^*}dx,
\end{equation}
as $n\rightarrow\infty$. We may suppose that
\begin{equation*}
\|\Psi_n\|_{X^\alpha}^2\rightarrow l~\text{and}~\displaystyle\int_{\mathbb{R}^N} g(x)|\Psi_n(x,0)|^{2_\alpha^*}dx\rightarrow l, ~\text{as}~n\rightarrow\infty,
\end{equation*}
for some $l\in[0,+\infty)$. If $l\neq0$, we know that $l\geq Sl^\frac{2}{2_\alpha^*}$ from Lemma \ref{l2.3}, and derive from (\ref{e2.24})-(\ref{e2.25}) and $\omega\in N_\lambda$ that
 \begin{align}
c&=I_\lambda(\omega)+\frac{1}{2}l-\frac{1}{2_\alpha^*}l\geq\alpha_\lambda^++\left(\frac{1}{2}-\frac{1}{2_\alpha^*}\right)l\geq\alpha_\lambda^++\frac{\alpha}{N}S^\frac{N}{2\alpha}.\notag
\end{align}
This yields a contradicting with the definition of $c$. Consequently, $l=0$, i.e. $\omega_n\rightarrow\omega$ in $X^\alpha(\mathbb{R}_+^{N+1})$.
\end{proof}

Finally, we shall show that a local minimum of $I_\lambda$ is achieved in $N_\lambda^+$, and that the minimizer is actually a positive solution of system (\ref{e2.7}).

\begin{theorem}\label{t2.1}
For each $\lambda\in(0,\Lambda_3),$ $\Lambda_3$ is the same as given in Lemma \ref{l2.9}, and $I_\lambda$ has a minimizer $\omega_\lambda^+$ in
$N_\lambda^+$ which satisfies:
\begin{enumerate}
  \item[\textup{(i)}]  $\omega_\lambda^+$ is a positive solution of system $(2.7)$.
  \item[\textup{(ii)}]  $I_\lambda(\omega_\lambda^+)=\alpha_\lambda^+$.
   \item[\textup{(iii)}]  $I_\lambda(\omega_\lambda^+)\rightarrow0$ as $\lambda\rightarrow0$.
    \item[\textup{(iv)}]  $\|\omega_\lambda^+\|_{X^\alpha}\rightarrow0$ as $\lambda\rightarrow0$.
\end{enumerate}
\end{theorem}

\begin{proof}

Following \cite[Lemma 4.7]{22Lawson},  we can obtain a $(PS)_{\alpha_\lambda^+}$-sequence for $I_\lambda$ defined by $\{\omega_n\}\subset N_\lambda$. By Lemma \ref{l2.10}, there exists a subsequence (still denoted by $\{\omega_n\}$) and $\omega_\lambda^+\in X^\alpha(\mathbb{R}_+^{N+1})$ such that
$\omega_n\rightarrow \omega_\lambda^+ $ in $X^\alpha(\mathbb{R}_+^{N+1})$ as $n\rightarrow\infty$. Since $N_\lambda^0=\emptyset$, we deduce that $\omega_\lambda^+\in N_\lambda^+$ and $I_\lambda(\omega_\lambda^+)=\alpha_\lambda^+<0$. It follows Lemma \ref{l2.6} that $\omega_\lambda^+$ is a solution of system (\ref{e2.7}). Using an analogous argument as the proof of \cite[Theorem 5.3]{20}, we can show that $\omega_\lambda^+$ is a positive solution of system (\ref{e2.7}). That is, the desired results (i) and (ii) are arrived.

From the proof of Lemma \ref{l2.8}, we have
\begin{equation*}
0>I_\lambda(\omega_\lambda^+)\geq-C\lambda^\frac{2}{2-q}.
\end{equation*}
This implies $I_\lambda(\omega_\lambda^+)\rightarrow0$ as $\lambda\rightarrow0^+$, i.e.  Part $(iii)$ holds.

In view of $\omega_\lambda^+\in N_\lambda^+$ and (\ref{e2.12}), we know that
\begin{equation}\label{e2.26}
\|\omega_\lambda^+\|_{X^\alpha}^2\leq\frac{4-q}{2}\displaystyle\int_{\mathbb{R}^N} f_\lambda|\omega(x,0)|^qdx\leq\lambda C|f_+|_{q^*}S^{-\frac{q}{2}}\|\omega_\lambda^+\|_{X^\alpha}^q.
\end{equation}
Since  $I_\lambda$ is coercive and bounded from below on $N_\lambda$, $\{\omega_\lambda^+\}_\lambda$ is bounded in $X^\alpha(\mathbb{R}_+^{N+1})$. It follows (\ref{e2.26}) that
\begin{equation*}
\|\omega_\lambda^+\|_{X^\alpha}^{2-q}\leq C\lambda^\frac{1}{2-q},
\end{equation*}
which implies $\|\omega_\lambda^+\|_{X^\alpha}\rightarrow0$ as $\lambda\rightarrow0^+$, i.e. Part (iv) holds.
\end{proof}

\section{\normalsize{Results on estimates}}\noindent

In this section, we shall give some useful estimates which will be used in the proof of Theorem 1.1.

For $b>0$, we set
\begin{equation*}
I_\infty^b(u)=\frac{1}{2}\|\omega\|_{X^\alpha}^2-\frac{b}{2_\alpha^*}\displaystyle\int_{\mathbb{R}^N} g(x)|\omega(x,0)|^{2_\alpha^*} dx
\end{equation*}
and
\begin{equation*}
N_\infty^b=\left\{u\in X^\alpha(\mathbb{R}_+^{N+1})\backslash\{0\};\,(I_\infty^b)'(\omega)\omega=0 \right\}.
\end{equation*}

\begin{lemma}\label{l3.1}
For each $\omega\in N_\lambda^-$, we have that
\begin{description}
  \item[\textup{(i)}] There is a unique $t_\omega^b$ such that $t_\omega^b\omega\in N_\infty^b$ and
  \begin{equation*}
\displaystyle\max_{t\geq0}I_\infty^b(t\omega)=I_\infty^b(t_\omega^b\omega)=\frac{\alpha}{N}b^{-\frac{N-2\alpha}{2\alpha}}\left(\frac{\|\omega\|_{X^\alpha}^{2_\alpha^*}}{\displaystyle\int_{\mathbb{R}^N} g(x)|\omega(x,0)|^{2_\alpha^*} dx}\right)^{\frac{N-2\alpha}{2\alpha}}.
\end{equation*}
  \item[\textup{(ii)}] For $\mu\in (0,1)$, there is a unique $t_\omega^1$ such that $t_\omega^1\omega\in N_\infty^1$,
  and
  \begin{equation*}
I_\infty^1(t_\omega^1\omega)\leq(1-\mu)^{-\frac{N}{2\alpha}}\left(I_\lambda(\omega)+\frac{2-q}{2q}\mu^\frac{q}{q-2}\lambda^\frac{2}{2-q}C\right).
\end{equation*}
\end{description}
\end{lemma}
\begin{proof}
(i) For each $\omega\in N_\lambda^-$, we define
\begin{equation*}
h(t)=I_\infty^b(t\omega)=\frac{t^2}{2}\|\omega\|_{X^\alpha}^2-\frac{t^{2_\alpha^*}}{2_\alpha^*}\displaystyle\int_{\mathbb{R}^N} bg(x)|\omega(x,0)|^{2_\alpha^*} dx.
\end{equation*}
Then $h(t)\rightarrow-\infty$ as $t\rightarrow\infty$, and
\begin{equation*}
h'(t)=t\|\omega\|_{X^\alpha}^2-t^{2_\alpha^*-1}\displaystyle\int_{\mathbb{R}^N} bg(x)|\omega(x,0)|^{2_\alpha^*} dx
\end{equation*}
and
\begin{equation*}
h''(t)=\|\omega\|_{X^\alpha}^2-(2_\alpha^*-1)t^{2_\alpha^*-2}\displaystyle\int_{\mathbb{R}^N} bg(x)|\omega(x,0)|^{2_\alpha^*} dx.
\end{equation*}

Define
\begin{equation*}
t_\omega^b=\left(\frac{\|\omega\|_{X^\alpha}^2}{\displaystyle\int_{\mathbb{R}^N} bg(x)|\omega(x,0)|^{2_\alpha^*} dx}\right)^\frac{1}{2_\alpha^*-2}>0.
\end{equation*}
It is easy to see that $h'(t_\omega^b)=0$ and $h''(t_\omega^b)=-(2_\alpha^*-2)\|\omega\|_{X^\alpha}^2<0$.
Thus there exists a unique $t_\omega^b>0$ such that $t_\omega^b\omega\in N_\infty^b$ and
  \begin{equation*}
\displaystyle\max_{t\geq0}I_\infty^b(t\omega)=I_\infty^b(t_\omega^b\omega)=\frac{\alpha}{N}b^{-\frac{N-2\alpha}{2\alpha}}\left(\frac{\|\omega\|_{X^\alpha}^{2_\alpha^*}}{\displaystyle\int_{\mathbb{R}^N} g(x)|\omega(x,0)|^{2_\alpha^*} dx}\right)^{\frac{N-2\alpha}{2\alpha}}.
\end{equation*}

(ii) For $\mu\in (0,1)$, it follows $(H_2)$ and Lemma \ref{l2.2} that
\begin{align}
\displaystyle\int_{\mathbb{R}^N}\lambda f_+|t_\omega^b\omega|^qdx&\leq \lambda |f_+|_{q^*}\|t_\omega^b\omega\|_{X^\alpha}^q\notag
\\&\leq\frac{2-q}{2}(\lambda C\mu^{-\frac{q}{2}})^\frac{2}{2-q}+\frac{q}{2}\left(\mu^\frac{q}{2}\|t_\omega^b\omega\|_{X^\alpha}^q\right)^\frac{2}{q}\notag
\\&=\frac{2-q}{2}\mu^\frac{q}{q-2}C\lambda^\frac{2}{2-q}+\frac{q\mu}{2}\|t_\omega^b\omega\|_{X^\alpha}^2.\notag
\end{align}
Combining with Part (i) with $b=\frac{1}{1-\mu}$ leads to
\begin{align}
I_\lambda(\omega)&=\displaystyle\max_{t\geq0}I_\lambda(t\omega)\nonumber\\&\geq I_\lambda(t_\omega^\frac{1}{1-\mu}\omega)\notag
 \\&\geq\frac{1-\mu}{2}\|t_\omega^\frac{1}{1-\mu}\omega)\|_{X^\alpha}^2+\frac{1}{4}(t_\omega^\frac{1}{1-\mu})^4\displaystyle\int_{\mathbb{R}^N} l(x)\phi_{\omega(x,0)}|\omega(x,0)|^2dx\notag
 \\&~~~-\frac{1}{2_\alpha^*}(t_\omega^\frac{1}{1-\mu})^{2_\alpha^*}\displaystyle\int_{\mathbb{R}^N}g(x)|\omega(x,0)|^{2_\alpha^*} dx-\frac{2-q}{2q}\mu^\frac{q}{q-2}C\lambda^\frac{2}{2-q}\notag
 \\&\geq(1-\mu)I_\infty^\frac{1}{1-\mu}(t_\omega^\frac{1}{1-\mu}\omega)-\frac{2-q}{2q}\mu^\frac{q}{q-2}C\lambda^\frac{2}{2-q}\notag
 \\&=(1-\mu)^\frac{N}{2\alpha}\frac{\alpha}{N}\left(\frac{\|\omega\|_{X^\alpha}^{2_\alpha^*}}{\displaystyle\int_{\mathbb{R}^N} g(x)|\omega(x,0)|^{2_\alpha^*} dx}\right)^{\frac{N-2\alpha}{2\alpha}}-\frac{2-q}{2q}\mu^\frac{q}{q-2}C\lambda^\frac{2}{2-q}\notag
 \\&=(1-\mu)^\frac{N}{2\alpha}I_\infty^1(t_\omega^1\omega)-\frac{2-q}{2q}\mu^\frac{q}{q-2}C\lambda^\frac{2}{2-q}.\notag
\end{align}
\end{proof}

Set $\omega_\varepsilon=E_\alpha(u_\varepsilon)$,  where $u_\varepsilon$ is defined in (\ref{e2.9}). Let $\eta(x,y)\in C^\infty(\mathbb{R}^{N}\times\mathbb{R})$ such that $0\leq\eta\leq1$, $|\nabla\eta|\leq C$ and
\begin{displaymath}
\eta(x,y):=\left\{\begin{array}{ll}
1,~ (x,y)\in B_\frac{r_0}{2}^+:=\{(x,y)\in \mathbb{R}_+^{N+1};\sqrt{|x_1|^2+|x_2|^2+\cdots+|x_N|^2+|y|^2}<\frac{r_0}{2},\\~~~~~~~~~~~~~~~~~~~~~~~~~~~y>0\},\\
0, ~ (x,y)\not\in B_{r_0}^+:=\{(x,y)\in \mathbb{R}_+^{N+1};\sqrt{|x_1|^2+|x_2|^2+\cdots+|x_N|^2+|y|^2}<r_0,\\~~~~~~~~~~~~~~~~~~~~~~~~~~~y>0\},
\end{array}\right.
\end{displaymath}
where $r_0$ is defined in Remark 1.1.

Set
\begin{equation}\label{e3.1}
 v_{\varepsilon,z}=\eta(x-z,y)\omega_\varepsilon(x-z,y),~~z\in M.
\end{equation}
We can deduce from\cite[Theorem 3.6]{17} or \cite[Lemma 2.4]{400} that
\begin{equation}\label{e3.2}
 \|v_{\varepsilon,z}\|_{\dot{X}^\alpha}^2=\|\omega_\varepsilon\|_{\dot{X}^\alpha}^2+O(\varepsilon^{N-2\alpha}),
\end{equation}
\begin{equation}\label{e3.3}
\displaystyle\int_{\mathbb{R}^N}|v_{\varepsilon,z}(x,0)|^qdx=\left\{\begin{array}{ll}
O(\varepsilon^\frac{2N-(N-2\alpha)q}{2}),~ &\text{if}~q>\frac{N}{N-2\alpha},\\
O(\varepsilon^\frac{(N-2\alpha)q}{2}),~~~~& \text{if}~q\leq\frac{N}{N-2\alpha},
\end{array}\right.
\end{equation}
and
\begin{equation}\label{e3.4}
\displaystyle\int_{\mathbb{R}^N}|v_{\varepsilon,z}(x,0)|^2dx=\left\{\begin{array}{ll}
O(\varepsilon^{2\alpha}\lg(\frac{1}{\varepsilon})),~ &\text{if}~N=4\alpha,\\
O(\varepsilon^{N-2\alpha}),~~~~~ &\text{if}~N<4\alpha.
\end{array}\right.
\end{equation}

\begin{lemma}\label{l3.2}
There holds
\begin{equation*}
\displaystyle\int_{\mathbb{R}^N}g(x) |v_{\varepsilon,z}(x,0)|^{2_\alpha^*}dx=\displaystyle\int_{\mathbb{R}^N}|\omega_\varepsilon(x,0)|^{2_\alpha^*}dx+O(\varepsilon^N).
\end{equation*}
\end{lemma}
\begin{proof}
By Remark 1.1 and the definition of $v_{\varepsilon,z}$, a straightforward calculation gives
\begin{align}
0&\leq\frac{1}{\varepsilon^N}\left[\displaystyle\int_{\mathbb{R}^N}|\omega_\varepsilon(x,0)|^{2_\alpha^*}dx-\displaystyle\int_{\mathbb{R}^N}g(x) |v_{\varepsilon,z}(x,0)|^{2_\alpha^*}dx\right]\notag
\\&=\displaystyle\int_{\mathbb{R}^N\backslash B_\frac{r_0}{2}}\frac{g(z)-g(x+z)\eta^{2_\alpha^*}(x,0)}{(\varepsilon^2+|x|^2)^N}dx+\displaystyle\int_{ B_\frac{r_0}{2}}\frac{g(z)-g(x+z)}{(\varepsilon^2+|x|^2)^N}dx\notag
\\&\leq\displaystyle\int_{\mathbb{R}^N\backslash B_\frac{r_0}{2}}\frac{2}{|x|^{2N}}dx+C\displaystyle\int_{B_\frac{r_0}{2}}\frac{|x|^\rho}{(\varepsilon^2+|x|^2)^N}dx\notag
\\&\leq\displaystyle\int_\frac{r_0}{2}^{+\infty}r^{-(N+1)}dr+\displaystyle\int_0^\frac{r_0}{2}r^{\rho-N-1}dr\nonumber\\&\leq C,\notag
\end{align}
for $z\in M$.
\end{proof}

Now, we introduce the definition of weak $L^r(\mathbb{R}^N)$ spaces.
\begin{definition}
Let $u$ be a measurable function on $\mathbb{R}^N$. We say the function $\mu_u$ defined by
\begin{equation*}
\mu_u(\lambda)=|\{x\in\mathbb{R}^N;\, |u(x)|>\lambda\}|,
\end{equation*}
the distribution function of $u$ (associated with the Lebesgue measure).
\end{definition}
\begin{definition}
For $1\leq r<\infty$, the weak $L^r(\mathbb{R}^N)$ space is defined as the set of all measurable functions $u$ such that
\begin{align}
|u|_{r,w}:&=\inf \left\{\nu>0;\mu_u(\lambda)\leq\frac{\nu^r}{\lambda^r}~\text{for~all}~\lambda>0 \right\}\notag
\\&=\sup\left\{\lambda\mu_u(\lambda)^\frac{1}{r};\lambda>0 \right\}\notag
\end{align}
is finite.
\end{definition}

We denote the weak $L^r(\mathbb{R}^N)$ spaces by $L^{r,w}(\mathbb{R}^N)$. As a consequence of
\begin{equation*}
\lambda^r\mu_u(\lambda)\leq\displaystyle\int_{\{x\in\mathbb{R}^N;|u(x)|>\lambda\}}|u|^rdx\leq\displaystyle\int_{\mathbb{R}^N}|u|^rdx,
\end{equation*}
we have $L^r(\Omega)\subset L^{r,w}(\mathbb{R}^N)$. It is easy to verify that $\frac{1}{|x|^{N-s}}\in L^{\frac{N}{N-s},w}(\mathbb{R}^N)$.

Next, we introduce a version of Young's inequality for $L^{r,w}(\mathbb{R}^N)$ spaces.
\begin{lemma}\label{l3.3}
Let $1<p,\,q,\,r<\infty$ satisfy $\frac{1}{r}+1=\frac{1}{p}+\frac{1}{q}$. Then there exists a constant $C>0$ such that
\begin{equation*}
|u*v|_r\leq C|u|_{p,w}|v|_q
\end{equation*}
holds for $u\in L^{p,w}(\mathbb{R}^N)$ and $v\in L^q(\mathbb{R}^N)$.
\end{lemma}

\begin{lemma}\label{l3.4}
There exists $\varepsilon_0>0$ small enough such that for $\varepsilon\in(0,\varepsilon_0)$, there holds
\begin{equation*}
\displaystyle\sup_{t\geq0}I_\lambda(\omega_\lambda^++tv_{\varepsilon,z})
<\alpha_\lambda^++\frac{\alpha}{N}S^\frac{N}{2\alpha}-\sigma(\varepsilon_0)
\end{equation*}
uniformly in $z\in M$. Furthermore, there exists $t_z^->0$ such that
\begin{equation*}
\omega_\lambda^++t_z^-v_{\varepsilon,z}\in N_\lambda^-
\end{equation*}
for $z\in M$.
\end{lemma}
\begin{proof}
Since
\begin{equation*}
\displaystyle\lim_{t\rightarrow0}I_\lambda(\omega_\lambda^++tv_{\varepsilon,z})=\alpha_\lambda^+<0~\text{and}~\displaystyle\lim_{t\rightarrow\infty}I_\lambda(\omega_\lambda^++tv_{\varepsilon,z})=-\infty,
\end{equation*}
for small $\varepsilon>0$, there exist a small $t_0>0$ and a large $t_1>0$  such that
\begin{equation}\label{e3.5}
I_\lambda(\omega_\lambda^++tv_{\varepsilon,z})<\alpha_\lambda^++\frac{\alpha}{N}S^\frac{N}{2\alpha},~~t\in(0,t_0]\cup[t_1,+\infty).
\end{equation}
We only need to prove that
\begin{equation*}
I_\lambda(\omega_\lambda^++tv_{\varepsilon,z})<\alpha_\lambda^++\frac{\alpha}{N}S^\frac{N}{2\alpha},~~t\in[t_0,t_1].
\end{equation*}

Note that $\omega_\lambda^+$ is a positive solution of system (\ref{e2.7}).
It follows Theorem \ref{t2.1}  that
\begin{align}\label{e3.6}
&I_\lambda(\omega_\lambda^++tv_{\varepsilon,z})\notag
\nonumber\\ =& I_\lambda(\omega_\lambda^+)+\frac{t^2}{2}\|v_{\varepsilon,z}\|_{X^\alpha}^2+t\displaystyle\int_{\mathbb{R}_+^{N+1}}y^{1-2\alpha}\nabla\omega_\lambda^+\nabla v_{\varepsilon,z} dxdy\notag
\nonumber\\&~~~~~+t\displaystyle\int_{\mathbb{R}^N}[l(x)\phi_{\omega_\lambda^+(x,0)}\omega_\lambda^+(x,0)v_{\varepsilon,z}(x,0)-g(x)|\omega_\lambda^+(x,0)|^{2_\alpha^*-1}v_{\varepsilon,z}(x,0)\nonumber\\
\nonumber\\&~~~~~~~~- f_\lambda(x)|\omega_\lambda^+(x,0)|^{q-1}v_{\varepsilon,z}(x,0)]dx\notag
\\&~~~~~+\frac{1}{4}\displaystyle\int_{\mathbb{R}^N}l(x)[\phi_{\omega_\lambda^+(x,0)+tv_{\varepsilon,z}(x,0)}(\omega_\lambda^++tv_{\varepsilon,z}(x,0))^2-\phi_{\omega^+(x,0)}(\omega_\lambda^+(x,0))^2\nonumber\\
\nonumber\\&~~~~~~~~-4\phi_{\omega_\lambda^+(x,0)}\omega_\lambda^+(x,0)(tv_{\varepsilon,z}(x,0))]dx\notag
\\&~~~~~-\frac{1}{2_\alpha^*}\displaystyle\int_{\mathbb{R}^N} g(x)[(\omega_\lambda^+(x,0)+tv_{\varepsilon,z}(x,0))^{2_\alpha^*}-(\omega_\lambda^+(x,0))^{2_\alpha^*}\nonumber\\
\nonumber\\&~~~~~~~~-2_\alpha^*(\omega_\lambda^+(x,0))^{2_\alpha^*-1}tv_{\varepsilon,z}(x,0)]dx\notag
\\&~~~~~-\frac{1}{q}\displaystyle\int_{\mathbb{R}^N} f_\lambda(x)[(\omega_\lambda^+(x,0)+tv_{\varepsilon,z}(x,0))^q-(\omega_\lambda^+(x,0))^q-q(\omega_\lambda^+(x,0))^{q-1}tv_{\varepsilon,z}(x,0)]dx\notag
\\ \leq & \alpha_\lambda^++\frac{t^2}{2}\|v_{\varepsilon,z}\|_{X^\alpha}^2-\frac{t^{2_\alpha^*}}{2_\alpha^*}\displaystyle\int_{\mathbb{R}^N} g(x)|v_\varepsilon(x,0)|^{2_\alpha^*}dx-t^{2_\alpha^*-1}\displaystyle\int_{\mathbb{R}^N} g(x)|v_\varepsilon(x,0)|^{2_\alpha^*-1}\omega_\lambda^+(x,0)dx\notag
\\&~~~~~+\frac{1}{4}\displaystyle\int_{\mathbb{R}^N}l(x)[\phi_{\omega_\lambda^+(x,0)+tv_{\varepsilon,z}(x,0)}(\omega_\lambda^+(x,0)+tv_{\varepsilon,z}(x,0))^2]dx\notag
\\&~~~~~-\frac{1}{4}\displaystyle\int_{\mathbb{R}^N}l(x)[\phi_{\omega_\lambda^+(x,0)}(\omega_\lambda^+(x,0))^2+4\phi_{\omega_\lambda^+(x,0)}\omega_\lambda^+(x,0)(tv_{\varepsilon,z}(x,0))]dx\notag
\\&~~~~~-\frac{1}{q}\displaystyle\int_{\mathbb{R}^N} f_\lambda(x)[(\omega_\lambda^+(x,0)+tv_{\varepsilon,z}(x,0))^q-(\omega_\lambda^+(x,0))^q-q(\omega_\lambda^+(x,0))^{q-1}tv_{\varepsilon,z}(x,0)]dx\nonumber\\&~~~~~+o(\varepsilon^\frac{N-2\alpha}{2})\notag
\\ \leq& \alpha_\lambda^++i(t)+j(t)-k(t)+o(\varepsilon^\frac{N-2\alpha}{2}),
\end{align}
for $z\in M$,
where
\begin{align}
i(t):=&\frac{t^2}{2}\|v_{\varepsilon,z}\|_{X^\alpha}^2-\frac{t^{2_\alpha^*}}{2_\alpha^*}\displaystyle\int_{\mathbb{R}^N} g(x)|v_\varepsilon(x,0)|^{2_\alpha^*}dx-t^{2_\alpha^*-1}\displaystyle\int_{\mathbb{R}^N} g(x)|v_\varepsilon(x,0)|^{2_\alpha^*-1}\omega_\lambda^+(x,0)dx,\notag
\\j(t):=&\frac{1}{4}\displaystyle\int_{\mathbb{R}^N}l(x)[\phi_{\omega_\lambda^+(x,0)+tv_{\varepsilon,z}(x,0)}(\omega_\lambda^+(x,0)+tv_{\varepsilon,z}(x,0))^2]dx\notag
\\&-\frac{1}{4}\displaystyle\int_{\mathbb{R}^N}l(x)[\phi_{\omega_\lambda^+(x,0)}(\omega_\lambda^+(x,0))^2+4\phi_{\omega_\lambda^+(x,0)}\omega_\lambda^+(x,0)(tv_{\varepsilon,z}(x,0))]dx,\notag
\\k(t):=&\frac{1}{q}\displaystyle\int_{\mathbb{R}^N} f_\lambda(x)[(\omega_\lambda^+(x,0)+tv_{\varepsilon,z}(x,0))^q-(\omega_\lambda^+(x,0))^q-q(\omega_\lambda^+(x,0))^{q-1}tv_{\varepsilon,z}(x,0)]dx.\notag
\end{align}

Since
\begin{align}
&\displaystyle\int_{\mathbb{R}^N} g(x)|v_{\varepsilon,z}(x,0)|^{2_\alpha^*-1}\omega_\lambda^+(x,0)dx\notag
\\=&\displaystyle\int_{\mathbb{R}^N} g(x+z)\eta^{2_\alpha^*-1}(x,0)|u_\varepsilon(x)|^{2_\alpha^*-1}\omega_\lambda^+(x+z,0)dx\notag
\\  \geq& C\varepsilon^\frac{N-2\alpha}{2}\displaystyle\int_{B_\frac{r_0}{2}}\frac{1}{(1+|x|^2)^\frac{N+2\alpha}{2}}dx\notag
\\ \geq& C\varepsilon^\frac{N-2\alpha}{2} \notag
\end{align}
for $z\in M$ and $C>0$,
it follows Lemma \ref{l3.2} and (\ref{e3.2})-(\ref{e3.4}) that
\begin{align}\label{e3.7}
i(t)&\leq\frac{t^2}{2}\|v_{\varepsilon,z}\|_{X^\alpha}^2-\frac{t^{2_\alpha^*}}{2_\alpha^*}\displaystyle\int_{\mathbb{R}^N} g(x)|v_{\varepsilon,z}(x,0)|^{2_\alpha^*}dx-C\varepsilon^\frac{N-2\alpha}{2}\notag
\\&\leq\frac{\alpha}{N}\left(\frac{\|v_{\varepsilon,z}\|_{\dot{X}^\alpha}^2}{\left(\displaystyle\displaystyle\int_{\mathbb{R}^N} g(x)|v_{\varepsilon,z}|^{2_\alpha^*}dx\right)^\frac{2}{2_\alpha^*}}\right)^\frac{N}{2\alpha}-C\varepsilon^\frac{N-2\alpha}{2}\notag
\\&\leq\frac{\alpha}{N}\left(\frac{\|\omega_\varepsilon\|_{\dot{X}^\alpha}^2+O(\varepsilon^{N-2\alpha})}{\left(\displaystyle\int_{\mathbb{R}^N}|\omega_\varepsilon(x,0)|^{2_\alpha^*}dx+O(\varepsilon^N)\right)^\frac{2}{2_\alpha^*}}\right)^\frac{N}{2\alpha}-C\varepsilon^\frac{N-2\alpha}{2}\notag
\\&\leq\frac{\alpha}{N}S^\frac{N}{2\alpha}-C\varepsilon^\frac{N-2\alpha}{2}
\end{align}
for $z\in M$ and $t\in[t_0,t_1]$.
A direct calculation gives
\begin{align}\label{e3.8}
j(t):=&\frac{1}{4}\displaystyle\int_{\mathbb{R}^N}l(x)[\phi_{\omega_\lambda^+(x,0)+tv_{\varepsilon,z}(x,0)}(\omega_\lambda^+(x,0)+tv_{\varepsilon,z}(x,0))^2]dx\notag
\\&-\frac{1}{4}\displaystyle\int_{\mathbb{R}^N}l(x)[\phi_{\omega_\lambda^+(x,0)}(\omega_\lambda^+(x,0))^2+4\phi_{\omega_\lambda^+(x,0)}\omega_\lambda^+(x,0)(tv_{\varepsilon,z}(x,0))]dx\notag
\\=&t\displaystyle\int_{\mathbb{R}^N}l(x)\omega_\lambda^+(x,0)\phi_{tv_{\varepsilon,z}(x,0)}v_{\varepsilon,z}(x,0) dx+\frac{t^2}{2}\displaystyle\int_{\mathbb{R}^N}l(x)\phi_{\omega_\lambda^+(x,0)}|v_{\varepsilon,z}(x,0)|^2dx\notag
\\&+\frac{t^2}{4}\displaystyle\int_{\mathbb{R}^N}l(x) \phi_{tv_{\varepsilon,z}(x,0)}|v_{\varepsilon,z}(x,0)|^2dx\notag
\\&+t^2\displaystyle\int_{\mathbb{R}^N\times\mathbb{R}^N}\frac{1}{|x-y|^{N-s}}l(y)\omega_\lambda^+(y,0)v_{\varepsilon,z}(y,0)l(x)\omega_\lambda^+(x,0)v_{\varepsilon,z}(x,0)dxdy.
\end{align}

For $t\in [t_0,t_1]$, by H\"{o}lder's inequality and (\ref{e3.3})-(\ref{e3.4}), we get
\begin{align}\label{e3.9}
\displaystyle\int_{\mathbb{R}^N}l(x)\omega_\lambda^+(x,0)\phi_{tv_{\varepsilon,z}(x,0)}v_{\varepsilon,z}(x,0) dx&\leq C|\phi_{tv_{\varepsilon,z}(x,0)}|_{2^*_\frac{s}{2}}|v_{\varepsilon,z}(x,0)|_\frac{2N}{N+s}\notag
\\&\leq C|v_{\varepsilon,z}(x,0)|_\frac{4N}{N+s}^2|v_{\varepsilon,z}(x,0)|_\frac{2N}{N+s},
\\ \displaystyle\int_{\mathbb{R}^N}l(x)\phi_{\omega_\lambda^+(x,0)}|v_{\varepsilon,z}(x,0)|^2dx&\leq C\displaystyle\int_{\mathbb{R}^N}|v_{\varepsilon,z}(x,0)|^2dx,
\\ \displaystyle\int_{\mathbb{R}^N}l(x) \phi_{tv_{\varepsilon,z}(x,0)}|v_{\varepsilon,z}(x,0)|^2dx&\leq C|\phi_{tv_{\varepsilon,z}(x,0)}|_{2^*_\frac{s}{2}}|v_{\varepsilon,z}(x,0)|^2_\frac{4N}{N+s}\leq C|v_{\varepsilon,z}(x,0)|^4_\frac{4N}{N+s}.
\end{align}
It follows  H\"{o}lder's inequality and Lemma \ref{l3.3} that
\begin{align}\label{e3.12}
&\displaystyle\int_{\mathbb{R}^N\times\mathbb{R}^N}\frac{1}{|x-y|^{N-s}}l(y)\omega_\lambda^+(y,0)v_{\varepsilon,z}(y,0)l(x)\omega_\lambda^+(x,0)v_{\varepsilon,z}(x,0)dxdy\notag
\\=&\displaystyle\int_{\mathbb{R}^N}l(x)\omega_\lambda^+(x,0)v_\varepsilon(x,0)\left(\displaystyle\int_{\mathbb{R}^N}\frac{1}{|x-y|^{N-s}}l(y)\omega_\lambda^+(y,0)v_\varepsilon(y,0)dy\right)dx\notag
\\ \leq& C|v_{\varepsilon,z}(x,0)|_\frac{2N}{N+s}\left|\frac{1}{|x|^{N-s}}\right|_{\frac{N}{N-s},w}|\omega_\lambda^+v_{\varepsilon,z}(x,0)|_\frac{2N}{N+s}\notag
\\ \leq& C\left|\frac{1}{|x|^{N-s}}\right|_{\frac{N}{N-s},w}|v_{\varepsilon,z}(x,0)|_\frac{2N}{N+s}^2\notag
\\ \leq& C|v_{\varepsilon,z}(x,0)|_\frac{2N}{N+s}^2.
\end{align}

Using the fact $\frac{2N}{N+s}<\frac{N}{N-2\alpha}$ and (\ref{e3.3}), we get
\begin{equation}\label{e3.13}
|v_{\varepsilon,z}(x,0)|_\frac{4N}{N+s}^2|v_{\varepsilon,z}(x,0)|_\frac{2N}{N+s}=o\left(\varepsilon^\frac{N-2\alpha}{2} \right).
\end{equation}
Since $2\alpha>\frac{N-2\alpha}{2}$ and $N-2\alpha>\frac{N-2\alpha}{2}$,  from (\ref{e3.4}) we have
\begin{equation}\label{e3.14}
\displaystyle\int_{\mathbb{R}^N}|v_{\varepsilon,z}(x,0)|^2dx=o\left(\varepsilon^\frac{N-2\alpha}{2} \right).
\end{equation}
If $\frac{4N}{N+s}\leq\frac{N}{N-2\alpha}$, from (\ref{e3.3}) we have
\begin{equation}\label{e3.15}
|v_{\varepsilon,z}(x,0)|^4_\frac{4N}{N+s}=O\left(\varepsilon^{2(N-2\alpha)}\right)=o\left(\varepsilon^\frac{N-2\alpha}{2} \right).
\end{equation}
If $\frac{4N}{N+s}>\frac{N}{N-2\alpha}$, then $4\left(\frac{2N-(N-2\alpha)q}{2q}\right)>\frac{N-2\alpha}{2}$ for $q=\frac{4N}{N+s}$. Using (\ref{e3.3}) again gives
\begin{equation}\label{e3.16}
|v_{\varepsilon,z}(x,0)|^4_\frac{4N}{N+s}=o\left(\varepsilon^\frac{N-2\alpha}{2} \right).
\end{equation}
Using (\ref{e3.8})-(\ref{e3.16}) implies that
\begin{equation}\label{e3.17}
j(t)=o\left(\varepsilon^\frac{N-2\alpha}{2}\right),
\end{equation}
for $z\in M$ and $t\in[t_0,t_1]$.

In view of the definition of $v_{\varepsilon,z}$, we know that
\begin{equation*}
\displaystyle\int_{\mathbb{R}^N} f_-(x)[(\omega_\lambda^+(x,0)+tv_{\varepsilon,z}(x,0))^q-(\omega_\lambda^+(x,0))^q-q(\omega_\lambda^+(x,0))^{q-1}tv_{\varepsilon,z}(x,0)]dx=0.
\end{equation*}
Then we have
\begin{align}\label{e3.18}
k(t)&=\frac{\lambda}{q}\displaystyle\int_{\mathbb{R}^N} f_+(x)[(\omega_\lambda^+(x,0)+tv_{\varepsilon,z}(x,0))^q-(\omega_\lambda^+(x,0))^q-q(\omega_\lambda^+(x,0))^{q-1}tv_{\varepsilon,z}(x,0)]dx\notag
\\&=\lambda\displaystyle\int_{\mathbb{R}^N} f_+(x)\left[\displaystyle\int_0^{tv_{\varepsilon,z}(x,0)}((\omega_\lambda^+(x,0)+s)^{q-1}-(\omega_\lambda^+(x,0))^{q-1})ds\right]dx\nonumber\\&\geq0
\end{align}
for  $z\in M$ and $t\in[t_0,t_1]$.

Substituting (\ref{e3.7}) and (\ref{e3.17})-(\ref{e3.18}) into (\ref{e3.6}) yields
\begin{equation}\label{e3.19}
I_\lambda(\omega_\lambda^++tv_{\varepsilon,z})<\alpha_\lambda^++\frac{\alpha}{N}S^\frac{N}{2\alpha}
-C\varepsilon^\frac{N-2\alpha}{2}+o\left(\varepsilon^\frac{N-2\alpha}{2}\right),
\end{equation}
for $z\in M$ and $t\in[t_0,\,t_1]$. It follows (\ref{e3.5}) and (\ref{e3.19}) that there exists small $\varepsilon_0>0$ such that for $\varepsilon\in(0,\varepsilon_0)$, we have
\begin{equation*}
\displaystyle\sup_{t\geq0}I_\lambda(\omega_\lambda^++tv_{\varepsilon,z})<\alpha_\lambda^++\frac{\alpha}{N}S^\frac{N}{2\alpha}
\end{equation*}
for  $z\in M$.

Similar to the proof of \cite[Lemma 4.4]{9},  there exists $t_z^->0$ such that
\begin{equation*}
\omega_\lambda^++t_z^-v_{\varepsilon,z}\in N_\lambda^-
\end{equation*}
for $z\in M$.
\end{proof}

\begin{lemma}\label{l3.5}
There holds
\begin{equation*}
\displaystyle\inf_{\omega\in N_\infty^1} I_\infty^1(\omega)=\displaystyle\inf_{\omega\in N^\infty} I^\infty(\omega)=\frac{\alpha}{N}S^\frac{N}{2\alpha},
\end{equation*}
where $I^\infty(\omega)=\frac{1}{2}\|\omega\|_{X^\alpha}^2-\frac{1}{2_\alpha^*}\displaystyle\int_{\mathbb{R}^N} |\omega(x,0)|^{2_\alpha^*}dx$ and $N^\infty=\left\{\omega\in X^\alpha(\mathbb{R}_+^{N+1}) \backslash\{0\};(I^\infty)'(\omega)\omega=0 \right\}$.
\end{lemma}
\begin{proof}
Observe that
\begin{equation*}
\displaystyle\max_{t\geq0}\left(\frac{a}{2}t^2-\frac{b}{2_\alpha^*}t^{2_\alpha^*}\right)=\frac{\alpha}{N}\left(\frac{a^2}{b^\frac{2}{2_\alpha^*}}\right)^\frac{2_\alpha^*}{2_\alpha^*-2}~\text{for}~a>0~\text{and}~b>0,
\end{equation*}
We deduce from Lemma \ref{l2.3} that
\begin{align}\label{e3.20}
\displaystyle\inf_{\omega\in N^\infty} I^\infty(\omega)&=\displaystyle\inf_{\omega\in X^\alpha(\mathbb{R}_+^{N+1}) \backslash\{0\}}\sup_{t\geq0}I^\infty(t\omega)\notag
\\&=\displaystyle\inf_{\omega\in X^\alpha(\mathbb{R}_+^{N+1}) \backslash\{0\}}\frac{\alpha}{N}\left(\frac{\|\omega\|_{X^\alpha}^2}{\left(\displaystyle\int_{\mathbb{R}^N} |\omega(x,0)|^{2_\alpha^*}dx\right)^\frac{2}{2_\alpha^*}}\right)^\frac{N}{2\alpha}\notag
\\&\geq \frac{\alpha}{N}S^\frac{N}{2\alpha}.
\end{align}
It follows from (\ref{e3.2}) and Lemma \ref{l3.2} that
\begin{align}
\displaystyle\sup_{t\geq0}I_\infty^1(tv_{\varepsilon,z})&=\frac{\alpha}{N}\left(\frac{\|v_{\varepsilon,z}\|_{X^\alpha}^2}{\left(\displaystyle\int_{\mathbb{R}^N} g(x)|v_{\varepsilon,z}(x,0)|^{2_\alpha^*}dx\right)^\frac{2}{{2_\alpha^*}}}\right)^\frac{N}{2\alpha}\notag
\\&=\frac{\alpha}{N}S^\frac{N}{2\alpha}+O\left(\varepsilon^{N-2\alpha}\lg\frac{1}{\varepsilon}\right).\notag
\end{align}
Thus we have
\begin{equation}\label{e3.21}
\displaystyle\inf_{u\in N_\infty^1} I_\infty^1(\omega)\leq\frac{\alpha}{N}S^\frac{N}{2\alpha},~\text{as}~\varepsilon\rightarrow0^+.
\end{equation}

Since $g(x)\leq1$, from (\ref{e3.20}) and (\ref{e3.21}) we have
\begin{align}
\frac{\alpha}{N}S^\frac{N}{2\alpha}&\leq\displaystyle\inf_{\omega\in N^\infty} I^\infty(\omega)\nonumber\\&=\displaystyle\inf_{\omega\in X^\alpha(\mathbb{R}_+^{N+1}) \backslash\{0\}}\sup_{t\geq0}I^\infty(t\omega)\notag
\\&\leq\displaystyle\inf_{\omega\in X^\alpha(\mathbb{R}_+^{N+1}) \backslash\{0\}}\sup_{t\geq0} I_\infty^1(t\omega)\notag
\\& =\displaystyle\inf_{\omega\in N_\infty^1} I_\infty^1(\omega)\nonumber\\&\leq\frac{\alpha}{N}S^\frac{N}{2\alpha}.\notag
\end{align}
\end{proof}

\section{\normalsize{Proof of Theorem \ref{t1.1}}}\noindent

In this section, we use the category theory to discuss multiple positive solutions of system (\ref{e2.7}) and prove Theorem \ref{s1.1}.

\begin{proposition} \cite{23} \label{p4.1}
  Let $\mathcal{M}$ be a $\mathcal{C}^{1,1}$ complete Riemannian manifold (modelled on a Hilbert space) and assume that $F\in \mathcal{C}^1(\mathbb{R},\mathbb{R})$ bounded from below. Let $-\infty<\displaystyle\inf_\mathcal{M} F<a<b<+\infty$. Suppose that $F$ satisfies the (PS) condition on the sublevel $\{u\in \mathcal{M}; F(u)\leq b\}$ and that $a$ is not a critical level for $F$. Then
  \begin{equation*}
  \sharp\{u\in F^a;\, \nabla F(u)=0\}\geq cat_{F^a}(F^a),
  \end{equation*}
  where $F^a\equiv\{u\in \mathcal{M}; F(u)\leq a\}$.
  \end{proposition}

\begin{proposition} \cite{23} \label{p4.2}
 Let $Q,\, \Omega^+$ and $\Omega^-$ be closed sets with $\Omega^-\subset\Omega^+$, and $\phi: Q\rightarrow\Omega^+,$ $\varphi: \Omega^-\rightarrow Q$ be two continuous maps such that $\phi\circ\varphi$ is homotopically equivalent to the embedding $j:\Omega^-\rightarrow\Omega^+$. Then $cat_Q(Q)\geq cat_{\Omega^+}(\Omega^-)$.
\end{proposition}

We will apply Propositions \ref{p4.1} and \ref{p4.2} to study the existence of multiple positive solutions of system (\ref{e2.7}).
Here is a scale operation on fractional Sobolev spaces.

\begin{lemma}\label{l4.1}
For any $\omega\in X^\alpha(\mathbb{R}_+^{N+1})$, given $\sigma>0$ and $x_0\in\mathbb{R}^N$, we consider the following scaled function
\begin{equation}
\rho(\omega)=\omega_\sigma:(x,y)\mapsto\sigma^\frac{N-\alpha}{2}\omega(\sigma(x-x_0),\sigma y).
\end{equation}
Then this scaling operation $\rho$ keeps constant norms $\|\omega_\sigma\|_{\dot{X}^\alpha}$ and $|\omega_\sigma(x,0)|_{2_\alpha^*}$ and is determined by the ``center" or ``concentration" point $x_0$ and the ``modulus" $\sigma$.
\end{lemma}
\begin{proof}
 We only need to prove that $\|\omega_\sigma\|_{\dot{X}^\alpha}=\|\omega\|_{\dot{X}^\alpha}$ and $|\omega_\sigma(x,0)|_{2_\alpha^*}=|\omega(x,0)|_{2_\alpha^*}$ as required.
 Indeed, let $z=\sigma(x-x_0)$ and $t=\sigma y$. Then $dz=\sigma^Ndx$ and $dt=\sigma dy$. Thus, we get
 \begin{equation*}
\displaystyle\int_{\mathbb{R}_+^{N+1}}y^{1-2\alpha}|\nabla\omega_\sigma|^2dxdy=\displaystyle\int_{\mathbb{R}_+^{N+1}}t^{1-2\alpha}|\nabla\omega|^2dzdt.
\end{equation*}
In an analogous manner, we can get $|\omega_\sigma(x,0)|_{2_\alpha^*}=|\omega(x,0)|_{2_\alpha^*}$ too.
\end{proof}

Based on Lemmas \ref{l2.3} and \ref{l4.1}, and \cite[Theorem 2.5]{24}, we can derive the following global compactness result immediately.

\begin{lemma}\label{l4.2}
Let $\{\omega_n\}\subset X^\alpha(\mathbb{R}_+^{N+1})\subset \dot{X}^\alpha(\mathbb{R}_+^{N+1})$ be a $(PS)$ sequence for $I^\infty$, where $I^\infty$ is given in Lemma \ref{l3.5}. Then there exist a number $k\in\mathbb{N}_0$, $k$ sequences of points $\{x_n^j\}\subset\mathbb{R}^N\ (1\leq j\leq k$), and $k+1$ sequences of functions $\{\omega_n^j\}\subset \dot{X}^\alpha(\mathbb{R}_+^{N+1})\ (0\leq j\leq k$) such that for some sequences, still denoted by $\{\omega_n\}$, there holds
\begin{equation*}
\omega_n(x,y)=\omega^0_n(x,y)+\displaystyle\Sigma_{j=1}^k\frac{1}{(\sigma_n^j)^\frac{N-\alpha}{2}}\omega_n^j\left(\frac{x-x_n^j}{\sigma_n^j},\frac{y}{\sigma_n^j}\right)
\end{equation*}
and
\begin{equation*}
\omega_n^j\rightarrow\omega^j~\text{in}~\dot{X}^\alpha(\mathbb{R}_+^{N+1}),~0\leq j\leq k,
\end{equation*}
where $\omega^0$ is a solution of
\begin{equation*}
\left\{\begin{array}{ll}
div(y^{1-2\alpha}\nabla\omega)=0,~ ~~~~~~&\text{in}~\mathbb{R}_+^{N+1},\\
-\frac{\partial\omega}{\partial\nu}=-\omega+|\omega|^{2_\alpha^*-2}\omega,~~&\text{on}~\mathbb{R}^N\times\{0\},
\end{array}\right.
\end{equation*}
and $\omega^j\ (1\leq j\leq k)$ are solutions of
\begin{equation*}
\left\{\begin{array}{ll}
div(y^{1-2\alpha}\nabla\omega)=0,~~~&\text{in}~\mathbb{R}_+^{N+1},\\
-\frac{\partial\omega}{\partial\nu}=|\omega|^{2_\alpha^*-2}\omega,~~~~~&\text{on}~\mathbb{R}^N\times\{0\}.
\end{array}\right.
\end{equation*}
Moreover, we have
\begin{itemize}
  \item If $x_n^j\rightarrow\overline{x}^j$, then either $\sigma_n^j\rightarrow+\infty$ or $\sigma_n^j\rightarrow0$.
  \item If $|x_n|\rightarrow+\infty$, then
  \begin{equation*}
\left\{\begin{array}{ll}
\sigma_n^j\rightarrow+\infty,\ \text{or}\\
\sigma_n^j\rightarrow0,\ \text{or}\\
\sigma_n^j\rightarrow\overline{\sigma}^j,~~0<\overline{\sigma}^j<+\infty.
\end{array}\right.
\end{equation*}
\end{itemize}

We further have
\begin{equation*}
\|\omega_n\|_{\dot{X}^\alpha}^2\rightarrow\displaystyle\Sigma_{j=0}^k\|\omega^j\|_{\dot{X}^\alpha}^2
\end{equation*}
and
\begin{equation*}
I^\infty(\omega_n)\rightarrow I^\infty(\omega^0)+\displaystyle\Sigma_{j=1}^k\widehat{I}^\infty(\omega^j),\ \text{as}\ n\rightarrow\infty,
\end{equation*}
 where $\widehat{I}^\infty(\omega^j)=\frac{1}{2}\|\omega^j\|_{\dot{X}^\alpha}^2-\frac{1}{2_\alpha^*}\displaystyle\int_{\mathbb{R}^N} |\omega^j(x,0)|^{2_\alpha^*}dx$ for $1\leq j\leq k$.
\end{lemma}

\begin{corollary}\label{c4.1}
Let $\{\omega_n\}\subset X^\alpha(\mathbb{R}_+^{N+1})$ be a nonnegative function sequence with $|\omega_n(x,0)|_{2_\alpha^*}=1$ and $\|\omega_n\|_{X^\alpha}^2\rightarrow S$. Then there exists a sequence $(x_n,\varepsilon_n)\in\mathbb{R}^N\times\mathbb{R}^+$ such that
\begin{equation*}
\omega_n(x,y):=\frac{1}{S^\frac{N-2\alpha}{4\alpha}}E_\alpha(u_{\varepsilon_n}(x-x_n))+o(1)
\end{equation*}
 in $\dot{X}^\alpha(\mathbb{R}_+^{N+1})$, where $u_\varepsilon$ is defined in (\ref{e2.9}). Moreover, if $x_n\rightarrow\overline{x}$ then $\varepsilon_n\rightarrow0$ or it is unbounded.
\end{corollary}

We define a continuous map $\Phi:X^\alpha(\mathbb{R}_+^{N+1})\backslash G\rightarrow\mathbb{R}^N$ by
\begin{equation*}
X^\alpha(\mathbb{R}_+^{N+1})\ni w\longmapsto\Phi(w)=\frac{\displaystyle\int_{\mathbb{R}^N} x|\omega(x,0)-\omega_\lambda^+(x,0)|^{2_\alpha^*}dx}{\displaystyle\int_{\mathbb{R}^N} |\omega(x,0)-\omega_\lambda^+(x,0)|^{2_\alpha^*}dx},
\end{equation*}
where $G=\left\{u\in X^\alpha(\mathbb{R}_+^{N+1});\ \displaystyle\int_{\mathbb{R}^N} |\omega(x,0)-\omega_\lambda^+(x,0)|^{2_\alpha^*}dx=0 \right\}$.

\begin{lemma}\label{l4.3}
For each $0<\delta<r_0$, there exist $\Lambda_\delta,\delta_0>0$ such that if $\omega\in N_\infty^1$, $I_\infty^1(\omega)<\frac{\alpha}{N}S^\frac{N}{2\alpha}+\delta_0$ and $\lambda\in(0,\Lambda_\delta)$, then $\Phi(\omega)\in M_\delta$, where $M_\delta$ is defined in Remark \ref{r1.1}.
\end{lemma}
\begin{proof}
Suppose the contrary. Then there exists a sequence $\{\omega_n\}\subset N_\infty^1$ such that $I_\infty^1(\omega_n)<\frac{\alpha}{N}S^\frac{N}{2\alpha}+o_n(1)$, while
$\lambda\rightarrow0^+$, and
\begin{equation}\label{e4.2}
\Phi(\omega_n)\not\in M_\delta ~\text{for}~n\in\mathbb{N}.
\end{equation}
Since
\begin{align}
1+\frac{\alpha}{N}S^\frac{N}{2}&>I_\infty^1(\omega_n)\notag
\\&=I_\infty^1(\omega_n)-\frac{1}{2_\alpha^*}(I_\infty^1)'(\omega_n)\omega_n\notag
\\&=\left(\frac{1}{2}-\frac{1}{2_\alpha^*}\right)\|\omega_n\|_{X^\alpha}^2,\notag
\end{align}
we see that $\{\omega_n\}$ is bounded in $X^\alpha(\mathbb{R}_+^{N+1})$.
According to Lemma \ref{l3.5},
 there is a sequence $\{t_n\}\subset\mathbb{R}^+$ such that $\{t_n \omega_n\}\in N^\infty$ and
\begin{align}
\frac{\alpha}{N}S^\frac{N}{2\alpha}\leq I^\infty(t_n\omega_n)\leq I_\infty^1(t_n\omega_n)
=\frac{\alpha}{N}t_n^2\|\omega_n\|_{X^\alpha}^2\leq I_\infty^1(\omega_n)=\frac{\alpha}{N}\|\omega_n\|_{X^\alpha}^2\leq\frac{\alpha}{N}S^\frac{N}{2\alpha}+o_n(1).\notag
\end{align}
Consequently, we have $t_n=1+o_n(1)$ as $n\rightarrow\infty$ and
\begin{align}\label{e4.3}
\displaystyle\lim_{n\rightarrow\infty}I^\infty(\omega_n)&=\displaystyle\lim_{n\rightarrow\infty}\frac{\alpha}{N}\|\omega_n\|_{X^\alpha}^2
\nonumber\\&=\displaystyle\lim_{n\rightarrow\infty}\frac{\alpha}{N}\displaystyle\int_{\mathbb{R}^N}  |\omega_n(x,0)|^{2_\alpha^*} dx\notag
\\&=\displaystyle\lim_{n\rightarrow\infty}\frac{\alpha}{N}\displaystyle\int_{\mathbb{R}^N}g(x)|\omega_n(x,0)|^{2_\alpha^*} dx
\nonumber\\&=\frac{\alpha}{N}S^\frac{N}{2\alpha}+o_n(1).
\end{align}

Set
\begin{equation*}
U_n=\frac{\omega_n}{\left(\displaystyle\int_{\mathbb{R}^N}|\omega_n(x,0)|^{2_\alpha^*} dx\right)^\frac{1}{2_\alpha^*}}.
\end{equation*}
Then $\displaystyle\int_{\mathbb{R}^N}|U_n(x,0)|^{2_\alpha^*} dx=1$. From (\ref{e4.3}) we deduce that
\begin{equation*}
\displaystyle\lim_{n\rightarrow\infty}\|U_n\|_{X^\alpha}^2=S,
\end{equation*}
which, together with (\ref{e4.3}) and Corollary \ref{c4.1}, implies that there exists a sequence $\{(x_n,\varepsilon_n)\}\subset\mathbb{R}^N\times\mathbb{R}^+$ such that
\begin{equation}\label{e4.4}
U_n(x,y):=\frac{1}{S^\frac{N-\alpha}{4}}E_\alpha(u_{\varepsilon_n}(x-x_n))+o_n(1).
\end{equation}
Moreover, when $n\rightarrow\infty$, $\{x_n\}\rightarrow\overline{x}$ or go to $\infty$. So we consider two cases here.

Case 1. Suppose that $\{x_n\}\rightarrow\infty$ as $n\rightarrow\infty$. From (\ref{e4.3}) and (\ref{e4.4}) we deduce that
\begin{align}
1&=\frac{\displaystyle\int_{\mathbb{R}^N}g(x)|\omega_n(x,0)|^{2_\alpha^*} dx}{\displaystyle\int_{\mathbb{R}^N}|\omega_n(x,0)|^{2_\alpha^*} dx}+o_n(1)\notag
\\&=\displaystyle\int_{\mathbb{R}^N}g(x)|U_n(x,0)|^{2_\alpha^*} dx+o_n(1)\notag
\\&=S^{-\frac{N}{2\alpha}}\displaystyle\int_{\mathbb{R}^N}g(x)|u_{\varepsilon_n}(x-x_n)|^{2_\alpha^*} dx+o_n(1)\notag
\\&=S^{-\frac{N}{2\alpha}}\displaystyle\int_{\mathbb{R}^N}g(x+x_n)|u_{\varepsilon_n}(x)|^{2_\alpha^*} dx+o_n(1)\notag
\\&=g_\infty,\notag
\end{align}
which contradicts the definition of $g_\infty$.

Case 2. Suppose that $\{x_n\}\rightarrow\overline{x}$ as $n\rightarrow\infty$. We deduce from Corollary \ref{c4.1} that $\varepsilon_n\rightarrow0$ as $n\rightarrow\infty$.  It follows  (\ref{e4.3}
) and (\ref{e4.4}) that
\begin{align}\label{e4.5}
1&=\frac{\displaystyle\int_{\mathbb{R}^N}g(x)|\omega_n(x,0)|^{2_\alpha^*} dx}{\displaystyle\int_{\mathbb{R}^N}|\omega_n(x,0)|^{2_\alpha^*} dx}+o_n(1)\notag
\\&=\displaystyle\int_{\mathbb{R}^N}g(x)|U_n(x,0)|^{2_\alpha^*} dx+o_n(1)\notag
\\&=S^{-\frac{N}{2\alpha}}\displaystyle\int_{\mathbb{R}^N}g(x)|u_{\varepsilon_n}(x-x_n)|^{2_\alpha^*} dx+o_n(1)\notag
\\&=S^{-\frac{N}{2\alpha}}\displaystyle\int_{\mathbb{R}^N}g(\sqrt{\varepsilon_n}x+x_n)|u_1(x)|^{2_\alpha^*} dx+o_n(1)\notag
\\&=g(\overline{x}),
\end{align}
where $u_1(x)=u_\varepsilon(x)$ for $\varepsilon=1$. (\ref{e4.5}) implies that $\overline{x}\in M$.
Moreover, we have
\begin{align}
\Phi(\omega_n)&=\frac{\displaystyle\int_{\mathbb{R}^N} x|\omega_n(x,0)-\omega_\lambda^+(x,0)|^{2_\alpha^*}dx}{\displaystyle\int_{\mathbb{R}^N} |\omega_n(x,0)-\omega_\lambda^+(x,0)|^{2_\alpha^*}dx}\notag
\\&=\frac{\displaystyle\int_{\mathbb{R}^N} x|\omega_n(x,0)|^{2_\alpha^*}dx}{\displaystyle\int_{\mathbb{R}^N} |\omega_n(x,0)|^{2_\alpha^*}dx}+o_\lambda(1),~\text{as}~\lambda\rightarrow0\notag
\\&=\frac{\displaystyle\int_{\mathbb{R}^N} x|U_n(x,0)|^{2_\alpha^*}dx}{\displaystyle\int_{\mathbb{R}^N} |U_n(x,0)|^{2_\alpha^*}dx}+o_\lambda(1)\nonumber\\&=\frac{\displaystyle\int_{\mathbb{R}^N} x|u_{\varepsilon_n}(x-x_n)|^{2_\alpha^*}dx}{\displaystyle\int_{\mathbb{R}^N} |u_{\varepsilon_n}(x-x_n)|^{2_\alpha^*}dx}+o_\lambda(1)\notag
\\&=\frac{\displaystyle\int_{\mathbb{R}^N} (x_n+\sqrt{\varepsilon_n}x)|u_1(x)|^{2_\alpha^*}dx}{\displaystyle\int_{\mathbb{R}^N} |u_1(x)|^{2_\alpha^*}dx}+o_\lambda(1)\notag
\\&\rightarrow \overline{x}\in M~\text{as}~n\rightarrow\infty,\notag
\end{align}
which is a contradiction with (\ref{e4.2}).
\end{proof}

\begin{lemma}\label{l4.4}
There exists $\lambda_\delta>0$ small enough such that if $\lambda\in(0,\lambda_\delta)$ and $\omega\in N_\lambda^-$ with $I_\lambda(\omega)<\frac{\alpha}{N}S^\frac{N}{2\alpha}+\frac{\delta_0}{2}$ ($\delta_0$ is given in Lemma \ref{l4.3}). Then we have $\Phi(\omega)\in M_\delta$.
\end{lemma}
\begin{proof}
For $\omega\in N_\lambda^-$ with $I_\lambda(\omega)<\frac{\alpha}{N}S^\frac{N}{2\alpha}+\frac{\delta_0}{2}$, we deduce from Lemma \ref{l3.1} $(ii)$ that there exists  a unique $t_\omega^1$ such that $t_\omega^1\omega\in N_\infty^1$ and
  \begin{equation*}
I_\infty^1(t_\omega^1\omega)\leq(1-\mu)^{-\frac{N}{2\alpha}}\left(I_\lambda(\omega)+\frac{2-q}{2q}\mu^\frac{q}{q-2}\lambda^\frac{2}{2-q}C\right)
\end{equation*}
for any $\mu\in(0,1)$.
Thus there exists small $\Lambda_\delta>0$ such that if $\lambda\in(0,\lambda_\delta)$, we have
 \begin{equation}\label{e4.6}
I_\infty^1(t_\omega^1\omega)\leq\frac{\alpha}{N}S^\frac{N}{2\alpha}+\delta_0.
\end{equation}
By (\ref{e4.6}) and Lemma \ref{l4.3},
we arrived at the desired result.
\end{proof}

Set $c_\lambda:=\alpha_\lambda^++\frac{\alpha}{N}S^\frac{N}{2\alpha}-\sigma(\varepsilon_0)$ and
\begin{equation*}
N_\lambda^-(c_\lambda):=\{\omega\in N_\lambda^-;I_\lambda(\omega)\leq c_\lambda\}.
\end{equation*}
\begin{lemma}\label{l4.5}
If $\omega$ is a critical point of $I_\lambda$ restricted on $N_\lambda^-$, then it is a critical point of $I_\lambda$ in $X^\alpha(\mathbb{R}_+^{N+1})$.
\end{lemma}
\begin{proof}
Let $\omega$ be  a critical point of $I_\lambda$ on $N_\lambda^-$, we have
\begin{equation*}
I'_\lambda(\omega)=\tau\Psi'_\lambda(\omega)
\end{equation*}
for some $\tau\in\mathbb{R}$, where $\Psi_\lambda$ is defined in (\ref{ee2.12}). Since $\omega\in N_\lambda^-$, we get
\begin{equation*}
0=I'_\lambda(\omega)\omega=\tau\Psi'_\lambda(\omega)\omega~\, \text{and}~\, \Psi'_\lambda(\omega)\omega<0,
\end{equation*}
which
 implies that $\tau=0$, i.e. $I'_\lambda(u)=0$.
\end{proof}

Denote by $I_{N_\lambda^-}$ the restriction of $I_\lambda$ on $N_\lambda^-$.

\begin{lemma}\label{l4.6}
$I_{N_\lambda^-}$ satisfies the $(PS)$ condition on $N_\lambda^-(c_\lambda)$.
\end{lemma}
\begin{proof}
Let $\{\omega_n\}\subset N_\lambda^-(c_\lambda)$ be a $(PS)$ sequence.
There exists a sequence $\{\theta_n\}\subset\mathbb{R}$ such that
\begin{equation*}
 I'_\lambda(\omega_n)=\theta_n\Psi'_\lambda(\omega_n)+o(1).
\end{equation*}
Since $\omega_n\in N_\lambda^-$,  we have
\begin{equation*}
  \Psi'_\lambda(\omega_n)\omega_n<0.
\end{equation*}
There exists a subsequence (still denoted by $\{\omega_n\}$) such that
\begin{equation*}
  \Psi'_\lambda(\omega_n)\omega_n\rightarrow l,\ l\leq0, \ \text{as}\ n \rightarrow \infty.
\end{equation*}

If $l=0$, we deduce from (\ref{e2.10}) and (\ref{e2.12}) that
\begin{align}
 I_\lambda(\omega_n)&=I_\lambda(\omega_n)-\frac{1}{q}I_\lambda'(\omega_n)\omega_n\notag
\\&=\left(\frac{1}{2}-\frac{1}{q}\right)\|\omega_n\|_{X^\alpha}^2+\left(\frac{1}{4}-\frac{1}{q}\right)\displaystyle\int_{\mathbb{R}^N} l(x)\phi_{\omega_n(x,0)} |\omega_n(x,0)|^2dx\notag
\\&-\left(\frac{1}{q}-\frac{1}{2_\alpha^*}\right)\displaystyle\int_{\mathbb{R}^N} g(x)|\omega_n(x,0)|^{2_\alpha^*}dx\notag
\\&=\frac{q-2}{2q}\|\omega_n\|_{X^\alpha}^2+\frac{q-4}{4q}\int_{\mathbb{R}^N} l(x)\phi_{\omega_n(x,0)} |\omega_n(x,0)|^2dx-\frac{2_\alpha^*-q}{2_\alpha^*q}\displaystyle\int_{\mathbb{R}^N} g(x)|\omega_n(x,0)|^{2_\alpha^*}\notag
\\&=\frac{(q-2)}{4q}\|\omega_n\|_{X^\alpha}^2+\frac{(q-2_\alpha^*)(2_\alpha^*-4)}{4q2_\alpha^*}\displaystyle\int_{\mathbb{R}^N} g(x)|\omega_n(x,0)|^{2_\alpha^*}+o_n(1)\nonumber\\&\leq0.\notag
\end{align}
This yields a contradiction with $\alpha_\lambda^->0$ (Lemma \ref{l2.8} $(ii)$). Then $l<0$. Due to  $I'_\lambda(\omega_n)\omega_n=0$, we conclude that $\{\theta_n\}\rightarrow0$ and
\begin{equation*}
I'_\lambda(\omega_n)\rightarrow0, \ \text{as}\ n \rightarrow \infty.
\end{equation*}
By virtue of Lemma \ref{l2.10}, we arrive at the desired result.
\end{proof}

We are now ready to prove Theorem \ref{t1.1}.

\begin{proof}[Proof of Theorem \ref{t1.1}]
Let $\delta,\lambda_\delta>0$ be given as in Lemmas \ref{l4.3} and \ref{l4.4}. To show that  $I_\lambda$ has at least
$cat_{M_\delta}(M)$ critical points in $N_\lambda^-(c_\lambda)$ for $\lambda\in(0,\lambda_\delta)$, for $z\in M$, by Lemma \ref{l3.4} we define
 \begin{equation*}
F(z)=\omega_\lambda^++t_z^-v_{\varepsilon,z}\in N_\lambda^-(c_\lambda).
\end{equation*}
It follows Lemma \ref{l4.4} that $\Phi(N_\lambda^-(c_\lambda))\subset M_\delta$ for $\lambda<\lambda_\delta$.
Define $\xi:[0,1]\times M\rightarrow M_\delta$ by
\begin{equation*}
[0,1]\times M\ni(\theta,z)\longmapsto\Phi\left(\omega_\lambda^++t_z^-v_{(1-\theta)\varepsilon,z}\right),
\end{equation*}
where $\Phi\left(\omega_\lambda^++t_z^-v_{(1-\theta)\varepsilon,z}\right)\in N_\lambda^-(c_\lambda).$ By a straightforward calculation, we have $\xi(0,z)=\Phi\circ F(z)$ and $\displaystyle\lim_{\theta\rightarrow1^-}\xi(\theta,z)=z$.
Hence, $\Phi\circ F$ is homotopic to $j:M\rightarrow M_\delta$. By virtue of Lemma \ref{l4.6} and Propositions \ref{p4.1} and \ref{p4.2}, we obtain that $I_{N_\lambda^-(c_\lambda)}$ has at least $cat_{M_\delta}(M)$ critical points in $N_\lambda^-(c_\lambda)$. According to Lemma \ref{l4.5}, we know that $I_\lambda$ has at least $cat_{M_\delta}(M)$ critical points  in $N_\lambda^-(c_\lambda)$. By an analogous argument as the proof of \cite[Theorem 5.3]{22Lawson}, we can see that system (\ref{e2.7}) admits at least $cat_{M_\delta}(M)$ positive solutions  in $N_\lambda^-(c_\lambda)$.
In view of Theorem \ref{t2.1} and  $N_\lambda^+\cap N_\lambda^-=\emptyset$, we arrive at the desired result.
\end{proof}

\section{\normalsize{Proof of Theorem \ref{t1.2}}}\noindent

Define the energy functional associated with system (\ref{e2.7}) by
\begin{align}
 I_{\lambda,g}(\omega)=&\frac{1}{2}\|\omega\|_{X^\alpha}^2+\frac{1}{4}\displaystyle\int_{\mathbb{R}^N}l(x)\phi_{\omega(x,0)}|\omega(x,0)|^2dx-\frac{\lambda}{q}\displaystyle\int_{\mathbb{R}^N}f(x)|\omega(x,0)|^qdx\nonumber\\&-\frac{1}{2_\alpha^*}\displaystyle\int_{\mathbb{R}^N}g(x)|\omega(x,0)|^{2_\alpha^*}dx,
\end{align}
where $\omega\in X^\alpha(\mathbb{R}_+^{N+1})$.

Set
 \begin{equation*}
 N_{\lambda,g}:=\left\{\omega\in X^\alpha(\mathbb{R}_+^{N+1})\backslash\{0\};、 I'_{\lambda,g}(\omega)\omega=0 \right\}.
\end{equation*}
 Then $\omega\in N_{\lambda,g}$ if and only if
 \begin{equation}
 \|\omega\|_{X^\alpha}^2+\displaystyle\int_{\mathbb{R}^N}l(x)\phi_{\omega(x,0)}|\omega(x,0)|^2dx-\lambda\displaystyle\int_{\mathbb{R}^N}f(x)|\omega(x,0)|^qdx-\displaystyle\int_{\mathbb{R}^N}g(x)|\omega(x,0)|^{2_\alpha^*}dx=0.
\end{equation}
For any $\omega\in N_{\lambda,g}$, we have
\begin{align}\label{e5.3}
I_{\lambda,g}(\omega)&=I_{\lambda,g}(\omega)-\frac{1}{q}I'_{\lambda,g}(\omega)\omega\notag\\
&=\left(\frac{1}{2}-\frac{1}{q}\right)\|\omega\|_{X^\alpha}^2+\left(\frac{1}{4}-\frac{1}{q}\right)\displaystyle\int_{\mathbb{R}^N}l(x)\phi_{\omega(x,0)}|\omega(x,0)|^2dx\notag
\\&~~~~+\left(\frac{1}{q}-\frac{1}{2_\alpha^*}\right)\displaystyle\int_{\mathbb{R}^N}g(x)|\omega(x,0)|^{2_\alpha^*}dx\nonumber\\
&\geq\left(\frac{1}{2}-\frac{1}{q}\right)\|\omega\|_{X^\alpha}^2\nonumber\\
&>0,
\end{align}
which implies that $I_{\lambda,g}$ is coercive and bounded from below on $N_{\lambda,g}$.

Let
 \begin{align}
 \Psi_{\lambda,g}(\omega):=&\|\omega\|_{X^\alpha}^2+\displaystyle\int_{\mathbb{R}^N}l(x)\phi_{\omega(x,0)}|\omega(x,0)|^2dx-\lambda\displaystyle\int_{\mathbb{R}^N}f(x)|\omega(x,0)|^qdx\nonumber\\
&-\displaystyle\int_{\mathbb{R}^N}g(x)|\omega(x,0)|^{2_\alpha^*}dx.
\end{align}
For $\omega\in N_{\lambda,g}$, we have
\begin{align}\label{e5.5}
\Psi'_{\lambda,g}(\omega)\omega\notag
&=2\|\omega\|_{X^\alpha}^2+4\displaystyle\int_{\mathbb{R}^N}l(x)\phi_{\omega(x,0)}|\omega(x,0)|^2dx-q \lambda\displaystyle\int_{\mathbb{R}^N}f(x)|\omega(x,0)|^qdx\nonumber\\
&~~~~-2_\alpha^*\displaystyle\int_{\mathbb{R}^N}g(x)|\omega(x,0)|^{2_\alpha^*}dx\notag
\\&=-2\|\omega\|_{X^\alpha}^2+(4-q) \displaystyle\int_{\mathbb{R}^N}f_\lambda(x)|\omega(x,0)|^qdx-(2_\alpha^*-4)\displaystyle\int_{\mathbb{R}^N}g(x)|\omega(x,0)|^{2_\alpha^*}dx\nonumber\\
&<0.
\end{align}

Define
\begin{equation*}
 \alpha_{\lambda,g}:=\displaystyle\inf_{\omega\in N_{\lambda,g}}I_{\lambda,g}(\omega)\geq0.
\end{equation*}
Then we have
\begin{lemma}\label{l5.1}
There holds
\begin{equation*}
 \alpha_{\lambda,g}\geq d_0>0
\end{equation*}
for some $d_0>0$.
\end{lemma}

\begin{proof}
From Lemma \ref{l2.2} and (\ref{e5.5}), we have
\begin{align}
2\|\omega\|_{X^\alpha}^2&\leq q\lambda\displaystyle\int_{\mathbb{R}^N}f(x)|\omega(x,0)|^qdx+2_\alpha^*\displaystyle\int_{\mathbb{R}^N}g(x)|\omega(x,0)|^{2_\alpha^*}dx\notag
\\&\leq\lambda C\|\omega\|_{X^\alpha}^q+C\|\omega\|_{X^\alpha}^{2_\alpha^*},
\end{align}
and thus
\begin{equation*}
 C\leq\lambda \|\omega\|_{X^\alpha}^{q-2}+\|\omega\|_{X^\alpha}^{2_\alpha^*-2}
\end{equation*}
for some $C>0$ independent of the choice of $\omega\in N_{\lambda,g}$. Using (\ref{e5.3}) implies the desired result.
\end{proof}

Since $2<4<2_\alpha^*$, we can directly obtain the following results.
\begin{lemma}\label{l5.2}
For each $\omega\in X^\alpha(\mathbb{R}_+^{N+1})\backslash\{0\}$, there exists a unique $t_\omega>0$ such that $t_\omega\omega\in N_{\lambda,g}$ and
\begin{equation*}
I_{\lambda,g}(t_\omega\omega)=\displaystyle\max_{t\geq0}I_{\lambda,g}(t\omega).
\end{equation*}
\end{lemma}
\begin{remark}
By Lemma \ref{l5.2}, it is obvious that
\begin{equation*}
\alpha_{\lambda,g}=\displaystyle\inf_{\omega\in X^\alpha(\mathbb{R}_+^{N+1})\backslash\{0\}}\max_{t\geq0}I_{\lambda,g}(t\omega).
\end{equation*}
Furthermore, we have
\begin{equation*}
0<\alpha_{\lambda_1,g}\leq\alpha_{\lambda_2,g}\leq\alpha_{0,g}
\end{equation*}
for  $\lambda_1\geq\lambda_2\geq0$.
\end{remark}

Consider an autonomous problem:
\begin{equation}
\left\{\begin{array}{ll}\label{e5.7}
(-\Delta)^\alpha u+u= \lambda f_\infty|u|^{q-2}u+g_\infty|u|^{2_\alpha^*-2}u,~ \text{in}~\mathbb{R}^N,\\
u\in H^\alpha(\mathbb{R}^N).
\end{array}\right.
\end{equation}
The harmonic extension of the problem (\ref{e5.7}) is:
\begin{equation}\label{e5.8}
\left\{\begin{array}{ll}
div(y^{1-2\alpha}\nabla\omega)=0,~ &\text{in}~\mathbb{R}_+^{N+1},\\
-\frac{\partial\omega}{\partial\nu}=-\omega+\lambda f_\infty|\omega|^{q-2}\omega+g_\infty|\omega|^{2_\alpha^*-2}\omega,~~&\text{in}~\mathbb{R}^N\times\{0\}.
\end{array}\right.
\end{equation}
The solutions of system (\ref{e5.8}) are precisely critical points of the energy functional defined by
\begin{equation*}
 I_{\lambda,\infty}(\omega)=\frac{1}{2}\|\omega\|_{X^\alpha}^2-\frac{\lambda}{q}\displaystyle\int_{\mathbb{R}^N}f_\infty|\omega(x,0)|^qdx-\frac{1}{2_\alpha^*}\displaystyle\int_{\mathbb{R}^N}g_\infty|\omega(x,0)|^{2_\alpha^*}dx,
\end{equation*}
where $\omega\in X^\alpha(\mathbb{R}_+^{N+1})$.

Define
\begin{equation*}
N_{\lambda,\infty}:=\{\omega\in X^\alpha(\mathbb{R}_+^{N+1})\backslash\{0\};I'_{\lambda,\infty}(\omega)\omega=0\}
\end{equation*}
and
\begin{equation*}
\alpha_{\lambda,\infty}:=\displaystyle\inf_{\omega\in N_{\lambda,\infty}}I_{\lambda,\infty}(\omega).
\end{equation*}
In order to give a precise description for the $(PS)$ condition  of $I_{\lambda,g}$, we recall the well-known concentration-compactness principle \cite{22}.

\begin{proposition}\label{p5.1}
Let $\rho_n(x)\in L^1(\mathbb{R}^N)$ be a non-negative sequence satisfying
\begin{equation*}
\displaystyle\lim_{n\rightarrow\infty}\displaystyle\int_{\mathbb{R}^3}\rho_n(x)dx=l,l>0.
\end{equation*}
Then there exists a subsequence, still denoted by $\{\rho_n(x)\}$, such that one of the following cases occurs.
\begin{description}
  \item[$(I)$] (Compactness) There exists $y_n\in \mathbb{R}^N$ such that for any $\varepsilon>0$ there exists $R>0$ such that
   \begin{equation*}
 \displaystyle\int_{B_R(y_n)}\rho_n(x)dx\geq l-\varepsilon,~~n=1,2,\ldots.
\end{equation*}
  \item[$(II)$] (Vanishing) For any fixed $R>0$, there holds
      \begin{equation*}
\displaystyle\lim_{n\rightarrow\infty}\sup_{y\in\mathbb{R}^N}\displaystyle\int_{B_R(y)}\rho_n(x)dx=0.
\end{equation*}
  \item[$(II)$] (Dichotomy) There exists $\alpha\in(0,l)$ such that for any $\varepsilon>0$, there exists $n_0\geq1$ and $\rho_n^{(1)}(x),\,\rho_n^{(2)}(x)\in L^1(\mathbb{R}^N)$, for $n\geq n_0$ there holds
       \begin{equation*}
|\rho_n-(\rho_n^{(1)}+\rho_n^{(2)})|_1<\varepsilon,~\left|\displaystyle\int_{\mathbb{R}^3}\rho_n^{(1)}(x)dx-\alpha\right|<\varepsilon,~\left|\displaystyle\int_{\mathbb{R}^3}\rho_n^{(2)}(x)dx-(l-\alpha)\right|<\varepsilon
\end{equation*}
and
\begin{equation*}
dist(supp\rho_n^{(1)},supp\rho_n^{(2)})\rightarrow\infty,~\text{as}~n\rightarrow\infty.
\end{equation*}
\end{description}
\end{proposition}

\begin{lemma}\label{l5.3}
$I_{\lambda,g}$ satisfies the $(PS)_c$ condition for  $c\in \left(0,\,\min\left\{\alpha_{\lambda,\infty},\frac{\alpha}{N}S^\frac{N}{2\alpha} \right\}\right)$.
\end{lemma}
\begin{proof}
Let $\{\omega_n\}\subset X^\alpha(\mathbb{R}_+^{N+1})$ be a $(PS)_c$ sequence for $I_{\lambda,g}$ with $c\in \left(0,\,\min\left\{\alpha_{\lambda,\infty},\frac{\alpha}{N}S^\frac{N}{2\alpha} \right\}\right)$. From (\ref{e5.3}), $\{\omega_n\}$ is bounded in $X^\alpha(\mathbb{R}_+^{N+1})$.

Set
\begin{equation*}
\rho_n(x)=|\omega_n(x,0)|^2+\phi_{\omega_n(x,0)}|\omega_n(x,0)|^2+\lambda f(x)|\omega_n(x,0)|^q+g(x)|\omega_n(x,0)|^{2_\alpha^*},
\end{equation*}
which belongs to $ L^1(\mathbb{R}^N).$
 We may assume that
 \begin{equation*}
|\rho_n(x)|_1\rightarrow l,\ l\geq0,\ \text{as}~n\rightarrow\infty.
\end{equation*}
Then we claim $l>0$. Otherwise, we get $\|\omega_n\|_{X^\alpha}\rightarrow0$ as $n\rightarrow\infty$, and then
 \begin{equation*}
I_{\lambda,g}(\omega_n)\rightarrow0~\text{as}~n\rightarrow\infty.
\end{equation*}
This obviously contradicts the hypothesis of $c>0$.

Suppose that $\{\rho_n\}$ vanishes. There is an $R>0$ such that
\begin{equation*}
\displaystyle\lim_{n\rightarrow\infty}\sup_{z\in\mathbb{R}^N}\int_{B_R(z)}|\omega_n(x,0)|^2dx=0.
\end{equation*}
By \cite[Lemma 3.3]{17}, we have $\omega_n(x,0)\rightarrow0$ in $L^r(\mathbb{R}^N)$ for $2<r<2_\alpha^*$. Then
\begin{equation*}
\displaystyle\int_{\mathbb{R}^N}f(x)|\omega_n(x,0)|^qdx\rightarrow0,~\text{as}~n\rightarrow\infty.
\end{equation*}
In view of $2<\frac{4N}{N+s}<2_\alpha^*$, we deduce from (\ref{e2.5}) that
\begin{equation*}
\left|\displaystyle\int_{\mathbb{R}^N}l(x)\phi_{\omega_n(x,0)}|\omega_n(x,0)|^2dx\right|\leq C|\omega_n(x,0)|_\frac{4N}{N+s}^4\rightarrow0,~\text{as}~n\rightarrow\infty.
\end{equation*}
Then we have
\begin{equation*}
\|\omega_n\|_{X^\alpha}^2-\displaystyle\int_{\mathbb{R}^N}g(x)|\omega_n(x,0)|^{2_\alpha^*}dx\rightarrow0,~~\text{as}~n\rightarrow\infty.
\end{equation*}

Suppose that
\begin{equation*}
\|\omega_n\|_{X^\alpha}^2\rightarrow l_1~\text{and}~\displaystyle\int_{\mathbb{R}^N}g(x)|\omega_n(x,0)|^{2_\alpha^*}dx\rightarrow l_1,~~\text{as}~n\rightarrow \infty
\end{equation*}
for some $l_1>0$. With the help of Lemma \ref{l2.3}, we get
\begin{equation*}
l_1\geq Sl_1^\frac{2}{2_\alpha^*}~\text{or}~l_1\geq S^\frac{N}{2\alpha}.
\end{equation*}
Consequently,
\begin{align}
c&=I_{\lambda,g}(\omega_n)+o(1)\notag
\\&=\frac{1}{2}\|\omega_n\|_{X^\alpha}^2-\frac{1}{2_\alpha^*}\displaystyle\int_{\mathbb{R}^N}g(x)|\omega_n(x,0)|^{2_\alpha^*}dx+o(1)\notag
\\&=\frac{\alpha}{N}l_1\nonumber\\
&\geq\frac{\alpha}{N}S^\frac{N}{2\alpha},\notag
\end{align}
which yields another contradiction with the definition of $c$.

Secondly, we suppose that the dichotomy occurs. It follows from Proposition \ref{p5.1} that for any $\varepsilon>0$, there exist $\alpha\in(0,l)$, $\{x_n\}\subset\mathbb{R}^N$ and $R_\varepsilon>0$ such that for any $R>R_\varepsilon$ and $\overline{R}>R_\varepsilon$ we have
\begin{equation*}
\displaystyle\liminf_{n\rightarrow\infty}\int_{B_R(x_n)}\rho_n(x)dx\geq\alpha-\varepsilon~\text{and}~\displaystyle\liminf_{n\rightarrow\infty}\displaystyle\int_{\mathbb{R}^N\backslash B_{\overline{R}}(x_n)}\rho_n(x)dx\geq(l-\alpha)-\varepsilon.
\end{equation*}
So there exist $\varepsilon_n\rightarrow0$, $R_n\rightarrow+\infty$ and $\overline{R_n}=4R_n$ such that
\begin{equation}\label{e5.9}
\displaystyle\int_{B_{R_n}(x_n)}\rho_n(x)dx\geq\alpha-\varepsilon_n~\ \text{and}~ \ \displaystyle\int_{\mathbb{R}^N\backslash B_{\overline{R_n}}(x_n)}\rho_n(x)dx\geq(l-\alpha)-\varepsilon_n.
\end{equation}
That is,
\begin{equation}\label{e5.10}
\displaystyle\int_{B_{4R_n}(x_n)\backslash B_{R_n}(x_n)}\rho_n(x)dx\leq2\varepsilon_n.
\end{equation}

Set $\eta_0(s)\in C^\infty(\mathbb{R}_+)$ and
\begin{equation*}
\eta_0(s)=\left\{\begin{array}{ll}
0,~ &\text{if}~s\leq1~\text{or}~s\geq4,\\
1,~~&\text{if}~2\leq s\leq3,
\end{array}\right.
\end{equation*}
and $|\eta'_0(s)|\leq2$.
Let $\eta_n(x,y)=\eta_0\left(\frac{|(x-x_n,y)|}{R_n}\right)$. It follows Lemma \ref{l2.4} $(ii)$ that
\begin{align}
&\displaystyle\int_{\mathbb{R}_+^{N+1}}y^{1-2\alpha}|\nabla\eta_n|^2|\omega_n(x,0)|^2dxdy\notag
\\ \leq& \frac{4}{R_n^2}\displaystyle\int_{B^+_{4R_n}(x_n,0)\backslash B^+_{R_n}(x_n,0)}y^{1-2\alpha}|\omega_n(x,0)|^2dxdy\notag
\\ \leq& \frac{C}{R_n^2}\left(\displaystyle\int_{B^+_{4R_n}(x_n,0)}y^{1-2\alpha}|\omega_n(x,0)|^{2\gamma}dxdy\right)^\frac{1}{\gamma}\left(\displaystyle\int_{B^+_{4R_n}(x_n,0)}y^{1-2\alpha}dxdy\right)^\frac{\gamma-1}{\gamma}\notag
\\
\leq& \frac{C}{R_n^2}\left(\displaystyle\int_{\mathbb{R}_+^{N+1}}y^{1-2\alpha}|\nabla\omega_n|^{2\gamma}dxdy\right)^\frac{1}{\gamma}(R_n^{N-1}R_n^{2-2\alpha})^\frac{\gamma-1}{\gamma}\notag
\\ \leq& CR_n^{-\frac{2}{N-2\alpha+2}}\nonumber\\
\leq & C
\end{align}
for some $C>0$, where $\gamma=1+\frac{2}{N-2\alpha}$. Thus, $\{\eta_n\omega_n\}$ is bounded in $X^\alpha(\mathbb{R}_+^{N+1})$, and
\begin{equation*}
I'_{\lambda,g}(\omega_n)(\eta_n\omega_n)\rightarrow0,~~{as}~n\rightarrow\infty.
\end{equation*}
We further have
\begin{align}\label{e5.12}
&\displaystyle\int_{\mathbb{R}_+^{N+1}}y^{1-2\alpha}\nabla\omega_n\nabla(\eta_n\omega_n)dxdy+\displaystyle\int_{\mathbb{R}^N}l(x)\phi_{\omega_n(x,0)}|\omega_n(x,0)|^2\eta_n(x,0)dx\nonumber\\
&+\displaystyle\int_{\mathbb{R}^N}|\omega_n(x,0)|^2\eta_n(x,0)dx\notag
\nonumber\\ =& \lambda\displaystyle\int_{\mathbb{R}^N}f(x)|\omega_n(x,0)|^q\eta_n(x,0)dx+\displaystyle\int_{\mathbb{R}^N}g(x)|\omega_n(x,0)|^{2_\alpha^*}\eta_n(x,0)dx+o_n(1).
\end{align}

On the other hand, it follows (\ref{e5.10}) that
\begin{align}
\displaystyle\int_{\mathbb{R}^N}g(x)|\omega_n(x,0)|^{2_\alpha^*}\eta_n(x,0)dx&\leq\displaystyle\int_{B_{4R_n}(x_n)\backslash B_{R_n}(x_n)}g(x)|\omega_n(x,0)|^{2_\alpha^*}dx=o_n(1),
\\ \displaystyle\int_{\mathbb{R}^N}f(x)|\omega_n(x,0)|^q\eta_n(x,0)dx&\leq\displaystyle\int_{B_{4R_n}(x_n)\backslash B_{R_n}(x_n)}f(x)|\omega_n(x,0)|^qdx=o_n(1),
\\ \displaystyle\int_{\mathbb{R}^N}l(x)\phi_{\omega_n(x,0)}|\omega_n(x,0)|^2\eta_n(x,0)dx&\leq\displaystyle\int_{B_{4R_n}(x_n)\backslash B_{R_n}(x_n)}l(x)\phi_{\omega_n(x,0)}|\omega_n(x,0)|^2dx=o_n(1),
\end{align}
and
\begin{align}
 \displaystyle\int_{\mathbb{R}^N}|\omega_n(x,0)|^2\eta_n(x,0)dx&\leq\displaystyle\int_{B_{4R_n}(x_n)\backslash B_{R_n}(x_n)}|\omega_n(x,0)|^2dx=o_n(1)\label{e5.16}.
\end{align}
It is not difficult to see that (\ref{e5.12})-(\ref{e5.16}) imply that
\begin{equation}\label{e5.17}
\displaystyle\int_{\mathbb{R}_+^{N+1}}y^{1-2\alpha}\nabla\omega_n\nabla(\eta_n\omega_n)dxdy=o_n(1).
\end{equation}

In view of Lemma \ref{e2.4} $(i)$ and $R_n\rightarrow\infty$ as $n\rightarrow\infty$, we have
\begin{align}\label{e5.18}
&\left|\displaystyle\int_{\mathbb{R}_+^{N+1}}y^{1-2\alpha}\omega_n\nabla\omega_n\nabla\eta_ndxdy\right|\notag
\\ \leq& \frac{C}{R_n}\displaystyle\int_{B^+_{4R_n}(x_n,0)\backslash B^+_{R_n}(x_n,0)}y^{1-2\alpha}|\omega_n\nabla\omega_n|dxdy\notag
\\ \leq& \frac{1}{R_n}\left(\displaystyle\int_{B^+_{4R_n}(x_n,0)\backslash B^+_{R_n}(x_n,0)}y^{1-2\alpha}|\omega_n|^2dxdy\right)^\frac{1}{2}\left(\displaystyle\int_{B^+_{4R_n}(x_n,0)}y^{1-2\alpha}|\nabla\omega_n|^2dxdy\right)^\frac{1}{2}\notag
\\ =& o_n(1).
\end{align}
We can derive from (\ref{e5.17}) and (\ref{e5.18}) that
\begin{equation}\label{e5.19}
\displaystyle\int_{\mathbb{R}_+^{N+1}}y^{1-2\alpha}\nabla\omega_n\nabla(\eta_n\omega_n)dxdy=\displaystyle\int_{\mathbb{R}_+^{N+1}}y^{1-2\alpha}|\nabla\omega_n|^2\eta_ndxdy=o_n(1).
\end{equation}

Set $\varphi_0(s)\in C^\infty(\mathbb{R}_+)$ and
\begin{equation*}
\varphi_0(s)=\left\{\begin{array}{ll}
0,~ &\text{if}~s\leq2,\\
1,~~&\text{if}~s\geq3,
\end{array}\right.
\end{equation*}
and $|\varphi'_0(s)|\leq2$. Let $\varphi_n(x,y)=\varphi_0\left(\frac{|(x-x_n,y)|}{R_n}\right)$ and
\begin{equation*}
U_n(x,y)=\varphi_n(x,y)\omega_n(x,y)~\text{and}~V_n(x,y)=(1-\varphi_n(x,y))\omega_n(x,y).
\end{equation*}
From (\ref{e5.9}) we have
\begin{equation}\label{e5.20}
\displaystyle\int_{\mathbb{R}^N}(|U_n(x,0)|^2+g(x)|U_n(x,0)|^{2_\alpha^*}+\lambda f(x)|U_n(x,0)|^q+l(x)\phi_{\omega_n(x,0)}|U_n(x,0)|^2)dx\geq\alpha-\varepsilon_n
\end{equation}
and
\begin{equation}\label{e5.21}
\displaystyle\int_{\mathbb{R}^N}(|V_n(x,0)|^2+g(x)|V_n(x,0)|^{2_\alpha^*}+\lambda f(x)|V_n(x,0)|^q+l(x)\phi_{\omega_n(x,0)}|V_n(x,0)|^2)dx\geq(l-\alpha)-\varepsilon_n.
\end{equation}
Using (\ref{e5.18}) and (\ref{e5.19}) leads to
\begin{align}\label{e5.22}
&\left|\displaystyle\int_{\mathbb{R}_+^{N+1}}y^{1-2\alpha}\nabla U_n\nabla V_ndxdy\right|\notag
\\ \leq& \left|\displaystyle\int_{B^+_{3R_n}(x_n,0)\backslash B^+_{2R_n}(x_n,0)}y^{1-2\alpha}|\nabla\omega_n|^2\varphi_n(1-\varphi_n)dxdy\right|\notag
\\&~~~~+\left|\frac{1}{R_n}\displaystyle\int_{B^+_{3R_n}(x_n,0)}y^{1-2\alpha}\omega_n\nabla\omega_n\nabla\varphi_n(1-2\varphi_n)dxdy\right|\notag
\\&~~~~~+\left|\frac{1}{R^2_n}\displaystyle\int_{B^+_{3R_n}(x_n,0)\backslash B^+_{2R_n}(x_n,0)}y^{1-2\alpha}|\nabla\varphi_n|^2|\omega_n|^2dxdy\right|\notag
\\ \leq& o_n(1)+\left|\frac{C}{R^2_n}\displaystyle\int_{\mathbb{R}_+^{N+1}}y^{1-2\alpha}|\omega_n|^2dxdy\right|\nonumber\\
=& o_n(1).
\end{align}
This implies that
\begin{equation}\label{e5.23}
\displaystyle\int_{\mathbb{R}_+^{N+1}}y^{1-2\alpha}|\nabla\omega_n|^2dxdy=\displaystyle\int_{\mathbb{R}_+^{N+1}}y^{1-2\alpha}|\nabla U_n|^2dxdy+\displaystyle\int_{\mathbb{R}_+^{N+1}}y^{1-2\alpha}|\nabla V_n|^2dxdy+o_n(1).
\end{equation}

From (\ref{e5.10}) we have
\begin{equation*}
\displaystyle\int_{\mathbb{R}^N}|U_n(x,0)V_n(x,0)|dx\leq\displaystyle\int_{B_{3R_n}(x_n)\backslash B_{2R_n}(x_n)}|\omega_n(x,0)|^2dx=o_n(1),
\end{equation*}
and then
\begin{equation}\label{e5.24}
\displaystyle\int_{\mathbb{R}_+^{N}}|\omega_n(x,0)|^2dx=\displaystyle\int_{\mathbb{R}_+^{N}}|U_n(x,0)|^2dx+\displaystyle\int_{\mathbb{R}_+^{N}}|V_n(x,0)|^2dx+o_n(1).
\end{equation}

Here we need to consider two cases.

Case 1. If $\{x_n\}$ is bounded, it follows from conditions $(A_2)$ and $(H_3)$ that
\begin{align}\label{e5.25}
\displaystyle\int_{\mathbb{R}^N}(f(x)-f_\infty)|V_n(x,0)|^qdx&\leq\displaystyle\sup_{|x-x_n|\geq2R_n}|f(x)-f_\infty|C\|\omega_n\|_{X^\alpha}^q\notag
\\&\leq C\displaystyle\sup_{|x-x_n|\geq2R_n}|f(x)-f_\infty|\nonumber\\&=o_n(1)
\end{align}
and
\begin{align}\label{e5.26}
\displaystyle\int_{\mathbb{R}^N}(g(x)-g_\infty)|V_n(x,0)|^{2_\alpha^*}dx&\leq\displaystyle\sup_{|x-x_n|\geq2R_n}|g(x)-g_\infty|C\|\omega_n\|_{X^\alpha}^{2_\alpha^*}\notag
\\&\leq C\displaystyle\sup_{|x-x_n|\geq2R_n}|g(x)-g_\infty|\nonumber\\&=o_n(1).
\end{align}
From (\ref{e5.10}) and (\ref{e5.26}) we have
\begin{align}\label{e5.27}
&\left|\displaystyle\int_{\mathbb{R}^N}g(x)|\omega_n(x,0)|^{2_\alpha^*-1}V_n(x,0)dx-\displaystyle\int_{\mathbb{R}^N}g_\infty|V_n(x,0)|^{2_\alpha^*}dx\right|\notag
\\&\leq\displaystyle\int_{\mathbb{R}^N}|g(x)-g_\infty||V_n(x,0)|^{2_\alpha^*}dx\nonumber\\&~~~~~~+\displaystyle\int_{\mathbb{R}^N}g(x)|V_n(x,0)|^{2_\alpha^*}[1-(1-\varphi_n(x,0))^{2_\alpha^*-1}]dx\notag
\\&\leq\displaystyle\int_{\mathbb{R}^N}|g(x)-g_\infty||V_n(x,0)|^{2_\alpha^*}dx+\displaystyle\int_{B_{3R_n}(x_n)\backslash B_{2R_n}(x_n)}|\omega_n(x,0)|^{2_\alpha^*}dx\nonumber\\&=o_n(1).
\end{align}
Similarly, by (\ref{e5.10}) and (\ref{e5.25}) we get
\begin{equation}
\left|\displaystyle\int_{\mathbb{R}^N}f(x)|\omega_n(x,0)|^{q-1}V_n(x,0)dx-\displaystyle\int_{\mathbb{R}^N}f_\infty|V_n(x,0)|^qdx\right|=o_n(1).
\end{equation}

Using (\ref{e5.10}) and (\ref{e5.22}), we deduce that
\begin{equation}
\left|\displaystyle\int_{\mathbb{R}_+^{N+1}}y^{1-2\alpha}(\nabla \omega_n\nabla V_n-|\nabla V_n|^2)dxdy\right|\leq\left|\displaystyle\int_{\mathbb{R}_+^{N+1}}y^{1-2\alpha}\nabla U_n\nabla V_ndxdy\right|=o_n(1),
\end{equation}
\begin{align}\label{e5.30}
\left|\displaystyle\int_{\mathbb{R}^N}(\omega_n(x,0)V_n(x,0)-|V_n(x,0)|^2)dx\right|&=\left|\displaystyle\int_{\mathbb{R}^N}U_n(x,0)V_n(x,0)dx\right|\notag
\\&\leq\displaystyle\int_{B_{3R_n}(x_n)\backslash B_{2R_n}(x_n)}|\omega_n(x,0)|^2dx\nonumber\\&=o_n(1),
\end{align}
and
\begin{align}\label{e5.31}
&\left|\displaystyle\int_{\mathbb{R}^N}l(x)\phi_{\omega_n(x,0)}\omega_n(x,0)V_n(x,0)dx-\displaystyle\int_{\mathbb{R}^N}l(x)\phi_{V_n(x,0)}|V_n(x,0)|^2dx\right|\notag
\\ =& \left|\displaystyle\int_{\mathbb{R}^N}l(x)\phi_{\omega_n(x,0)}|V_n(x,0)|^2dx+\displaystyle\int_{\mathbb{R}^N}l(x)\phi_{\omega_n(x,0)}U_n(x,0)V_n(x,0)dx \right.\nonumber \\&
 \left.~~~~~-\displaystyle\int_{\mathbb{R}^N}l(x)\phi_{V_n(x,0)}|V_n(x,0)|^2dx\right|\notag
\\ \leq& \left|\displaystyle\int_{\mathbb{R}^N}l(x)(\phi_{\omega_n(x,0)}-\phi_{V_n(x,0)})|V_n(x,0)|^2dx\right|+\left|\displaystyle\int_{\mathbb{R}^N}l(x)\phi_{\omega_n(x,0)}U_n(x,0)V_n(x,0)dx\right|\notag
\\\leq &\left|\displaystyle\int_{B_{3R_n}(x_n)\backslash B_{2R_n}(x_n)}l(x)\phi_{V_n(x,0)}|V_n(x,0)|^2dx\right|\nonumber
\\&~~~+\left|\displaystyle\int_{B_{3R_n}(x_n)\backslash B_{2R_n}(x_n)}l(x)\phi_{\omega_n(x,0)}|V_n(x,0)|^2dx\right|\notag
\\&~~~+\left|\displaystyle\int_{B_{3R_n}(x_n)}l(x)\phi_{\omega_n(x,0)}U_n(x,0)V_n(x,0)dx\right|\notag
\\ \leq & C\left|\displaystyle\int_{B_{3R_n}(x_n)\backslash B_{2R_n}(x_n)}l(x)\phi_{\omega_n(x,0)}|\omega_n(x,0)|^2dx\right|\nonumber\\
=& o_n(1).
\end{align}
It follows (\ref{e5.27})-(\ref{e5.31}) that
\begin{equation}\label{e5.32}
I'_\infty(V_n)V_n=I'_{\lambda,g}(\omega_n)V_n=o(1),
\end{equation}
where
\begin{align*}
I_\infty(V_n)=&\frac{1}{2}\|V_n\|_{X^\alpha}^2+\frac{1}{4}\displaystyle\int_{\mathbb{R}^N}l(x)\phi_{V_n(x,0)}|V_n(x,0)|^2dx-\frac{\lambda}{q}\displaystyle\int_{\mathbb{R}^N}f_\infty|V_n(x,0)|^qdx\\&-\frac{1}{2_\alpha^*}\displaystyle\int_{\mathbb{R}^N}g_\infty|V_n(x,0)|^{2_\alpha^*}dx.
\end{align*}

Combining (\ref{e5.20})-(\ref{e5.21}) and (\ref{e5.23})-(\ref{e5.24}), we know that $\|V_n\|_{X^\alpha}^2<\|\omega_n\|_{X^\alpha}^2$. Hence, there exists $t_n<1$ such that $t_nV_n\in N_\infty$ and
\begin{align}
\alpha_{\lambda,\infty}\leq\alpha_\infty&\leq I_\infty(t_nV_n)\notag
\\&=I_\infty(t_nV_n)-\frac{1}{q}I'_\infty(t_nV_n)(t_nV_n)\notag
\\&=\left(\frac{1}{2}-\frac{1}{q}\right)\|t_nV_n\|_{X^\alpha}^2+\left(\frac{1}{4}-\frac{1}{q}\right)\displaystyle\int_{\mathbb{R}^N}l(x)\phi_{t_nV_n(x,0)}|t_nV_n(x,0)|^2dx\notag
\\&~~~~~+\left(\frac{1}{q}-\frac{1}{2_\alpha^*}\right)\displaystyle\int_{\mathbb{R}^N}g_\infty|t_nV_n(x,0)|^{2_\alpha^*}dx\notag
\\&<\left(\frac{1}{2}-\frac{1}{q}\right)\|V_n\|_{X^\alpha}^2+\left(\frac{1}{4}-\frac{1}{q}\right)\displaystyle\int_{\mathbb{R}^N}l(x)\phi_{V_n(x,0)}|V_n(x,0)|^2dx\notag
\\&~~~~~+\left(\frac{1}{q}-\frac{1}{2_\alpha^*}\right)\displaystyle\int_{\mathbb{R}^N}g_\infty|V_n(x,0)|^{2_\alpha^*}dx\notag
\\&<\left(\frac{1}{2}-\frac{1}{q}\right)\|\omega_n\|_{X^\alpha}^2+\left(\frac{1}{4}-\frac{1}{q}\right)\displaystyle\int_{\mathbb{R}^N}l(x)\phi_{\omega_n(x,0)}|\omega_n(x,0)|^2dx\notag
\\&~~~~~+\left(\frac{1}{q}-\frac{1}{2_\alpha^*}\right)\displaystyle\int_{\mathbb{R}^N}g(x)|\omega_n(x,0)|^{2_\alpha^*}dx\notag
\\&=I_{\lambda,g}(\omega_n)-\frac{1}{q}I'_{\lambda,g}(\omega_n)(\omega_n)\notag
\\&=c+o_n(1)\nonumber\\&<\alpha_{\lambda,\infty},\notag
\end{align}
where $N_\infty:=\{\omega\in X^\alpha(\mathbb{R}_+^{N+1})\{0\};I'_\infty(\omega)\omega=0\}$ and $\alpha_\infty:=\displaystyle\inf_{\omega\in N_\infty}I_\infty(\omega)$.
This yields a contradiction.

Case 2. If $\{x_n\}$ is unbounded,  we assume that $|x_n|\rightarrow+\infty$ as $n\rightarrow\infty$. According to condition $(H_3)$, for any $\varepsilon>0$ there exists an $R'>0$ such that
\begin{equation*}
|g(x)-g_\infty|<\varepsilon,~~\text{for~all}~|x|\geq R'.
\end{equation*}
Set $3R_n=|x_n|-R'$. Then $R_n\rightarrow+\infty$  and $B_{3R_n}(x_n)\subset\mathbb{R}^N\backslash B_{R'}(0)$ as $n\rightarrow\infty$. So we have
\begin{align}
&\displaystyle\int_{\mathbb{R}^N}(g(x)-g_\infty)|U_n(x,0)|^{2_\alpha^*}dx\notag
\\=& \displaystyle\int_{B_{3R_n}(x_n)}(g(x)-g_\infty)|\varphi_n(x,0) \omega_n(x,0)|^{2_\alpha^*}dx\notag
\\ \leq& \displaystyle\int_{\mathbb{R}^N\backslash B_{R'}(0)}(g(x)-g_\infty)|\omega_n(x,0|^{2_\alpha^*}dx\notag
\\ \leq& C\varepsilon.\notag
\end{align}
Similarly, from condition $(A_1)$  we have
\begin{equation*}
\left|\displaystyle\int_{\mathbb{R}^N}(f(x)-f_\infty)|U_n(x,0)|^qdx\right|\leq C\varepsilon.
\end{equation*}
Similar to the proof of (\ref{e5.32}), we obtain
\begin{equation}\label{e5.33}
I'_\infty(U_n)U_n=\langle I'_{\lambda,g}(\omega_n)U_n+o(1)=o(1).
\end{equation}

Since $\|U_n\|_{X^\alpha}^2<\|\omega_n\|_{X^\alpha}^2$,  there exists $t_n<1$ such that $t_nU_n\in N_\infty$
and\begin{align}
\alpha_{\lambda,\infty}\leq & \alpha_\infty\leq I_\infty(t_nU_n)\notag
\\ =& I_\infty(t_nV_n)-\frac{1}{q}I'_\infty(t_nU_n)(t_nU_n)\notag
\\=& \left(\frac{1}{2}-\frac{1}{q}\right)\|t_nU_n\|_{X^\alpha}^2+\left(\frac{1}{4}-\frac{1}{q}\right)\displaystyle\int_{\mathbb{R}^N}l(x)\phi_{t_nU_n(x,0)}|t_nU_n(x,0)|^2dx\notag
\\&~~~+\left(\frac{1}{q}-\frac{1}{2_\alpha^*}\right)\displaystyle\int_{\mathbb{R}^N}g_\infty|t_nV_n(x,0)|^{2_\alpha^*}dx\notag
\\<& \left(\frac{1}{2}-\frac{1}{q}\right)\|V_n\|_{X^\alpha}^2+\left(\frac{1}{4}-\frac{1}{q}\right)\displaystyle\int_{\mathbb{R}^N}l(x)\phi_{U_n(x,0)}|U_n(x,0)|^2dx\notag
\\&~~~~+\left(\frac{1}{q}-\frac{1}{2_\alpha^*}\right)\displaystyle\int_{\mathbb{R}^N}g_\infty|U_n(x,0)|^{2_\alpha^*}dx\notag
\\<& \left(\frac{1}{2}-\frac{1}{q}\right)\|\omega_n\|_{X^\alpha}^2+\left(\frac{1}{4}-\frac{1}{q}\right)\displaystyle\int_{\mathbb{R}^N}l(x)\phi_{\omega_n(x,0)}|\omega_n(x,0)|^2dx\notag
\\&~~~~+\left(\frac{1}{q}-\frac{1}{2_\alpha^*}\right)\displaystyle\int_{\mathbb{R}^N}g(x)|\omega_n(x,0)|^{2_\alpha^*}dx\notag
\\=& I_{\lambda,g}(\omega_n)-\frac{1}{q}I'_{\lambda,g}(\omega_n)(\omega_n)\notag
\nonumber\\ =& c+o(1)
\nonumber\\<& \alpha_{\lambda,\infty}.\notag
\end{align}
This obviously is a contradiction.

Hence, dichotomy can not happen and  $\{\rho_n(x)\}$ is compact, i.e., there exists $\{x_n\}\subset\mathbb{R}^N$ such that for any $\varepsilon>0$ there is an $R>0$ such that
\begin{equation}\label{e5.34}
\displaystyle\int_{B^c_R(x_n)}\rho_n(x)dx<\varepsilon,
\end{equation}
where $B^c_R(x_n)=\mathbb{R}^N\backslash B_R(x_n)$.

We claim that $\{x_n\}$ is bounded. Otherwise,  we assume that $|x_n|\rightarrow+\infty$ as $n\rightarrow\infty$. Let $|x_n|\geq R+R'$ for large $n$. Then we  have
\begin{align}
&\displaystyle\int_{\mathbb{R}^N}(g(x)-g_\infty)|\omega_n(x,0)|^{2_\alpha^*}dx\notag
\\ =& \displaystyle\int_{B_{R}(x_n)}(g(x)-g_\infty)|\omega_n(x,0)|^{2_\alpha^*}dx+\displaystyle\int_{B^c_{R}(x_n)}(g(x)-g_\infty)|\omega_n(x,0)|^{2_\alpha^*}dx\notag
\\ \leq& C\varepsilon+\displaystyle\int_{B_{R'}(0)}(g(x)-g_\infty)|\omega_n(x,0)|^{2_\alpha^*}dx
\nonumber\\ \leq& C\varepsilon.\notag
\end{align}
Similarly, we get
\begin{equation*}
\displaystyle\int_{\mathbb{R}^N}(f(x)-f_\infty)|\omega_n(x,0)|^qdx\leq C\varepsilon
\end{equation*}
and
\begin{equation*}
I_{\lambda,g}(\omega_n)=I_\infty(\omega_n)+o(1)~\text{and}~\langle I'_\infty(\omega_n)\omega_n=o(1).
\end{equation*}
Furthermore, there exists $\{t_n\}\rightarrow1$ such that $t_n\omega_n\in N_\infty$ and
\begin{equation*}
c=I_{\lambda,g}(\omega_n)+o(1)=I_\infty(t_n\omega_n)\geq \alpha_\infty\geq\alpha_{\lambda,\infty},
\end{equation*}
which contradicts the definition of $c$.

Note that $\{\omega_n\}$ is bounded in $X^\alpha(\mathbb{R}_+^{N+1})$. There exists $\omega\in X^\alpha(\mathbb{R}_+^{N+1})$ such that
$\omega_n\rightharpoonup \omega$ in $X^\alpha(\mathbb{R}_+^{N+1})$.
It follows  Lemma \ref{l2.2} and (\ref{e5.34}) that
\begin{equation*}
\displaystyle\int_{\mathbb{R}^N} f(x)|\omega_n(x,0)|^qdx\rightarrow\displaystyle\int_{\mathbb{R}^N} f(x)|\omega(x,0)|^qdx,~\text{as}~n\rightarrow\infty.
\end{equation*}
Applying H\"{o}lder's inequality, Lemma \ref{l2.1} $(vi)$, Lemmas \ref{l2.2} and (\ref{e2.5}) as well as (\ref{e5.34}), we obtain
\begin{align}
&\left|\displaystyle\int_{\mathbb{R}^N}l(x)\phi_{\omega_n(x,0)}|\omega_n(x,0)|^2dx-\displaystyle\int_{\mathbb{R}^N}l(x)\phi_{\omega(x,0)}|\omega(x,0)|^2dx\right|\notag\\
&\leq|l|_\infty\displaystyle\int_{B_R}\phi_{\omega_n(x,0)-\omega(x,0)}|\omega_n(x,0)-\omega(x,0)|^2dx+c\varepsilon\nonumber\\&\leq C|\omega_n(x,0)-\omega(x,0)|_\frac{4N}{N+s}^4\nonumber\\&\rightarrow0,
~\text{as}~n\rightarrow\infty. \notag
\end{align}

Setting $\Psi_n=\omega_n-\omega$ and using the Br\'{e}zis-Lieb Lemma, we know that
\begin{itemize}
  \item $\|\Psi_n\|_{X^\alpha}^2=\|\omega_n\|_{X^\alpha}^2-\|\omega\|_{X^\alpha}^2+o_n(1)$,
  \item $\displaystyle\int_{\mathbb{R}^N} g(x)|\Psi_n(x,0)|^{2_\alpha^*}dx=\displaystyle\int_{\mathbb{R}^N} g(x)|\omega_n(x,0)|^{2_\alpha^*}dx-\displaystyle\int_{\mathbb{R}^N} g(x)|\omega(x,0)|^{2_\alpha^*}dx+o_n(1)$.
\end{itemize}
According to Lemma \ref{l2.1} $(iii)$, it is easy to see that $I'_{\lambda,g}(\omega)=0$. Thus we have
\begin{equation}\label{e5.35}
\frac{1}{2}\|\Psi_n\|_{X^\alpha}^2-\frac{1}{2_\alpha^*}\displaystyle\int_{\mathbb{R}^N} g(x)|\Psi_n(x,0)|^{2_\alpha^*}dx=c-I_{\lambda,g}(\omega)+o_n(1)
\end{equation}
and
\begin{equation}\label{e5.36}
o(1)=I'_{\lambda,g}(\omega_n)\Psi_n=(I'_{\lambda,g}(\omega_n)-I'_{\lambda,g}(\omega))\Psi_n=\|\Psi_n\|_{X^\alpha}^2-\displaystyle\int_{\mathbb{R}^N} g(x)|\Psi_n(x,0)|^{2_\alpha^*}dx,
\end{equation}
as $n\rightarrow\infty$.

We may suppose that
\begin{equation*}
\|\Psi_n\|_{X^\alpha}^2\rightarrow l_2~\text{and}~\displaystyle\int_{\mathbb{R}^N} g(x)|\Psi_n(x,0)|^{2_\alpha^*}dx\rightarrow l_2, ~\text{as}~n\rightarrow\infty
\end{equation*}
for some $l_2\in[0,+\infty)$. If $l_2=0$, we obtain that $\omega_n\rightarrow\omega$ in $X^\alpha(\mathbb{R}_+^{N+1})$ directly.
If $l_2\neq0$, since $l_2\geq Sl_2^\frac{2}{2_\alpha^*}$, using Lemma \ref{l2.3} and (\ref{e5.35})-(\ref{e5.36}), for $\omega\in N_{\lambda,g}$ we have
 \begin{align}
c&=I_{\lambda,g}(\omega)+\frac{1}{2}l_2-\frac{1}{2_\alpha^*}l_2\geq\left(\frac{1}{2}-\frac{1}{2_\alpha^*}\right)l_2\geq\frac{\alpha}{N}S^\frac{N}{2\alpha},\notag
\end{align}
which contradicts the definition of $c$. Consequently, the only choice is  $l_2=0$, i.e. $\omega_n\rightarrow\omega$ in $X^\alpha \left(\mathbb{R}_+^{N+1}\right)$.
\end{proof}

In the following, we shall give some estimates which are crucially used in the proof of Theorem \ref{t1.2}.

\begin{lemma}
There exists $\lambda_0>0$ such that if $\lambda\in(0,\lambda_0)$, then
\begin{equation*}
\frac{\alpha}{N}S^\frac{N}{2\alpha}<\alpha_{\lambda,\infty}.
\end{equation*}
\end{lemma}
\begin{proof}
By way of contradiction, otherwise there exists a sequence $\{\lambda_n\}\rightarrow0$ such that
\begin{equation*}
\alpha_{\lambda_n,\infty}\leq\frac{\alpha}{N}S^\frac{N}{2\alpha},~\text{as}~n\rightarrow\infty.
\end{equation*}
For each $\lambda_n>0$, it follows \cite[Theorem 3.6]{17} that there exists $\omega_n\in X^\alpha(\mathbb{R}_+^{N+1})$ such that
\begin{equation}\label{e5.37}
I_{\lambda_n,\infty}(\omega_n)=\alpha_{\lambda_n,\infty}~\text{and}~ I'_{\lambda_n,\infty}(\omega_n)\omega_n=0.
\end{equation}
So we have
 \begin{align}
 \frac{\alpha}{N}S^\frac{N}{2\alpha}&\geq \alpha_{\lambda_n,\infty}\notag\\
&= I_{\lambda_n,\infty}(\omega_n)-\frac{1}{q}I'_{\lambda_n,\infty}(\omega_n)\omega_n\notag\\
&=\left(\frac{1}{2}-\frac{1}{q}\right)\|\omega_n\|_{X^\alpha}^2+\left(\frac{1}{q}-\frac{1}{2_\alpha^*}\right)\displaystyle\int_{\mathbb{R}^N}g_\infty|\omega_n(x,0)|^{2_\alpha^*}dx\notag\\
&\geq\left(\frac{1}{2}-\frac{1}{q}\right)\|\omega_n\|_{X^\alpha}^2,\notag
\end{align}
which implies that $\{\omega_n\}$ is bounded in $ X^\alpha(\mathbb{R}_+^{N+1})$. Using Lemma \ref{l2.2} gives
\begin{equation}\label{e5.38}
\displaystyle\int_{\mathbb{R}^N} \lambda_nf_\infty|\omega_n(x,0)|^qdx\leq \lambda_n C\|\omega_n\|_{X^\alpha}^q\rightarrow0,~\text{as}~n\rightarrow\infty.
\end{equation}

Based on (\ref{e5.37}) and (\ref{e5.38}), we assume that
 \begin{equation*}
\|\omega_n\|_{X^\alpha}^2\rightarrow l_3~\text{and}~\displaystyle\int_{\mathbb{R}^N} g_\infty|\omega_n(x,0)|^{2_\alpha^*}dx\rightarrow l_3, ~\text{as}~n\rightarrow\infty
\end{equation*}
for some $l_3\in[0,+\infty)$. If $l_3=0$, then we get
 \begin{equation*}
\alpha_{0,\infty}=0.
\end{equation*}
 This is impossible since $I_{0,\infty}$ has a mountain pass geometry.
If $l_3\neq0$, in view of $l_3\geq S\left(\frac{l_3}{g_\infty}\right)^\frac{2}{2_\alpha^*}$ and $g_\infty<1$ by Remark \ref{r1.1},  we have
\begin{align}
 \frac{\alpha}{N}S^\frac{N}{2\alpha}&\geq \alpha_{\lambda_n,\infty}\notag\\
&= I_{\lambda_n,\infty}(\omega_n)-\frac{1}{2_\alpha^*}I'_{\lambda_n,\infty}(\omega_n)\omega_n+o_n(1)\notag\\
&=\frac{\alpha}{N}\|\omega_n\|_{X^\alpha}^2+o(1)\notag
\\&=\frac{\alpha}{N}S^\frac{N}{2\alpha}g_\infty^{-\frac{N-2\alpha}{2\alpha}}\nonumber\\&>\frac{\alpha}{N}S^\frac{N}{2\alpha}.\notag
\end{align}
This is a contradiction.
\end{proof}

\begin{lemma}\label{l5.5}
There exist small $\varepsilon_0>0$ and $\sigma(\varepsilon_0)>0$ such that for $\varepsilon\in(0,\varepsilon_0)$ we have
\begin{equation*}
\displaystyle\sup_{t\geq0}I_{\lambda,g}(tv_{\varepsilon,z})<\frac{\alpha}{N}S^\frac{N}{2\alpha}-\sigma(\varepsilon_0)~\text{uniformly~with}~ z\in M,
\end{equation*}
where $v_{\varepsilon,z}$ is defined in (\ref{e3.1}). Furthermore, there exists $t_z>0$ such that
\begin{equation*}
t_zv_{\varepsilon,z}\in N_{\lambda,g}~~\text{for~each}~z\in M.
\end{equation*}
\end{lemma}

\begin{proof}
Since
\begin{equation*}
\displaystyle\lim_{t\rightarrow0}I_{\lambda,g}(tv_{\varepsilon,z})=\alpha_\lambda^+<0~\, \text{and}\, ~\displaystyle\lim_{t\rightarrow\infty}I_{\lambda,g}(tv_{\varepsilon,z})=-\infty
\end{equation*}
for  $z\in M$ and small $\varepsilon>0$,  there exist small $t_0>0$  and large $t_1>0$ such that
\begin{equation}\label{e5.39}
I_{\lambda,g}(tv_{\varepsilon,z})<\frac{\alpha}{N}S^\frac{N}{2\alpha},~~\text{for}~t\in(0,t_0]\cup[t_1,+\infty).
\end{equation}

We only need to show that
\begin{equation*}
I_{\lambda,g}(tv_{\varepsilon,z})<\frac{\alpha}{N}S^\frac{N}{2\alpha}
\end{equation*}
for $z\in M$ and $t\in[t_0,t_1]$.
By (\ref{e2.5}) and (\ref{e3.2}),  it follows Lemma \ref{l3.2} and  Theorem \ref{t2.1}  that
\begin{align}\label{e5.40}
&I_{\lambda,g}(tv_{\varepsilon,z})\notag
\nonumber \\ \leq& \frac{t^2}{2}\|v_{\varepsilon,z}\|_{X^\alpha}^2-\lambda C\displaystyle\int_{\mathbb{R}^N}f(x)|v_{\varepsilon,z}(x,0)|^qdx-\frac{t^{2_\alpha^*}}{2_\alpha^*}\displaystyle\int_{\mathbb{R}^N} g(x)|v_\varepsilon(x,0)|^{2_\alpha^*}dx+C|v_\varepsilon(x,0)|_\frac{4N}{N+s}^4\notag
\nonumber\\ \leq& \frac{\alpha}{N}\left(\frac{\|v_{\varepsilon,z}\|_{\dot{X}^\alpha}^2}{\left(\displaystyle\displaystyle\int_{\mathbb{R}^N} g(x)|v_{\varepsilon,z}|^{2_\alpha^*}dx\right)^\frac{2}{2_\alpha^*}}\right)^\frac{N}{2\alpha}+C|v_\varepsilon(x,0)|^2+C|v_\varepsilon(x,0)|_\frac{4N}{N+s}^4\nonumber\\&~~~~-\lambda C\displaystyle\int_{\mathbb{R}^N}f(x)|v_{\varepsilon,z}(x,0)|^qdx\notag
\nonumber\\ \leq& \frac{\alpha}{N}\left(\frac{\|\omega_\varepsilon\|_{\dot{X}^\alpha}^2+O(\varepsilon^{N-2\alpha})}{\left(\displaystyle\int_{\mathbb{R}^N}|\omega_\varepsilon(x,0)|^{2_\alpha^*}dx+O(\varepsilon^N)\right)^\frac{2}{2_\alpha^*}}\right)^\frac{N}{2\alpha}+C|v_\varepsilon(x,0)|^2+C|v_\varepsilon(x,0)|_\frac{4N}{N+s}^4\nonumber\\&~~~~-\lambda C\displaystyle\int_{\mathbb{R}^N}|v_{\varepsilon,z}(x,0)|^qdx\notag
\nonumber\\ \leq& \frac{\alpha}{N}S^\frac{N}{2\alpha}+C|v_\varepsilon(x,0)|^2+C|v_\varepsilon(x,0)|_\frac{4N}{N+s}^4-\lambda C\displaystyle\int_{\mathbb{R}^N}|v_{\varepsilon,z}(x,0)|^qdx+O(\varepsilon^{N-2\alpha})
\end{align}
for $z\in M$ and $t\in[t_0,t_1]$.

We need to consider the following two cases.

Case 1. If $\frac{4N}{N+s}>\frac{N}{N-2\alpha}$, it follows (\ref{e3.2})-(\ref{e3.3}) and (\ref{e5.40}) that
\begin{equation*}
I_{\lambda,g}(tv_{\varepsilon,z})\leq\frac{\alpha}{N}S^\frac{N}{2\alpha}+O(\varepsilon^{4\alpha+s-N})+O(\varepsilon^{N-2\alpha})-\lambda C\varepsilon^\frac{2N-(N-2\alpha)q}{2}
\end{equation*}
for  $z\in M$ and small $\varepsilon>0$.

Case 2. If $\frac{4N}{N+s}<\frac{N}{N-2\alpha}$, it follows  (\ref{e3.2})-(\ref{e3.3}) and (\ref{e5.40}) that
\begin{align*}
I_{\lambda,g}(tv_{\varepsilon,z})&\leq\frac{\alpha}{N}S^\frac{N}{2\alpha}+O(\varepsilon^{2(N-2\alpha)})+O(\varepsilon^{N-2\alpha})-\lambda C\varepsilon^\frac{2N-(N-2\alpha)q}{2}\nonumber\\&<\frac{\alpha}{N}S^\frac{N}{2\alpha}-\lambda C\varepsilon^\frac{2N-(N-2\alpha)q}{2}
\end{align*}
for $z\in M$ and small $\varepsilon>0$.

In  either case,  there exists small $\varepsilon_0>0$  such that for $\varepsilon\in(0,\varepsilon_0)$ we have
\begin{equation*}
\displaystyle\sup_{t\geq0}I_{\lambda,g}(tv_{\varepsilon,z})<\frac{\alpha}{N}S^\frac{N}{2\alpha}-\sigma(\varepsilon_0)
\end{equation*}
for  $z\in M$. Moreover, it follows  Lemma \ref{l5.2} that there exists $t_z^->0$ such that
$t_z^-v_{\varepsilon,z}\in N_\lambda^-$
for each $z\in M$.
\end{proof}

Define a continuous map $\Psi:X^\alpha(\mathbb{R}^{N+1}_+)\backslash \{0\}\rightarrow\mathbb{R}^N$ by
\begin{equation*}
X^\alpha(\mathbb{R}^{N+1}_+)\ni w\longmapsto\Psi(\omega):=\frac{\displaystyle\int_{\mathbb{R}^N} x|\omega(x,0)|^{2_\alpha^*}dx}{\displaystyle\int_{\mathbb{R}^N} |\omega(x,0)|^{2_\alpha^*}dx}.
\end{equation*}
Similar to the proof of Lemma \ref{l4.3}, we can obtain the following Lemma.
\begin{lemma}\label{l5.6}
For each $0<\delta<r_0$, there exists $\overline{\delta_0}>0$ such that if $\omega\in N_\infty^1$ and $I_\infty^1(\omega)<\frac{\alpha}{N}S^\frac{N}{2\alpha}+\overline{\delta_0}$, then $\Psi(\omega)\in M_\delta$.
\end{lemma}

\begin{lemma}\label{l5.7}
There exists small $\overline{\lambda_\delta}>0$  such that if $\lambda\in(0,\overline{\lambda_\delta})$ and $\omega\in N_{\lambda,g}$ with $I_{\lambda,g}(\omega)<\frac{\alpha}{N}S^\frac{N}{2\alpha}+\frac{\overline{\delta_0}}{2}$\ ($\overline{\delta_0}$ is given in Lemma \ref{l5.6}), then $\Phi(\omega)\in M_\delta$.
\end{lemma}

\begin{proof}
For $\omega\in N_{\lambda,g}$ with $I_{\lambda,g}(\omega)<\frac{\alpha}{N}S^\frac{N}{2\alpha}+\frac{\overline{\delta_0}}{2}$,
it follows (\ref{e5.3}) and Lemma \ref{l5.1} that
\begin{equation}\label{ee5.41}
0<\|\omega\|_{X^\alpha}\leq C.
\end{equation}
Then there exists  $t_\omega>0$ such that $t_\omega\omega\in N_\infty^1$. Now we can conclude that $t_\omega\leq C$ for some $C>0$ independent of $\omega$. Otherwise, there exists a sequence $\{t_{\omega_n}\}\rightarrow\infty$ as $n\rightarrow\infty$ such that
\begin{equation}\label{e5.42}
\displaystyle\int_{\mathbb{R}^N} g(x)|\omega_n(x,0)|^{2_\alpha^*}dx=t_{\omega_n}^{2-2_\alpha^*}\|\omega_n\|_{X^\alpha}^2\rightarrow0,~\text{as}~n\rightarrow\infty.
\end{equation}

Since $\omega_n\in N_{\lambda,g}$, from (\ref{ee5.41}) and (\ref{e5.42}) we have
\begin{align}
\|\omega_n\|_{X^\alpha}^2&=\lambda\displaystyle\int_{\mathbb{R}^N}f(x)|\omega_n(x,0)|^qdx+\displaystyle\int_{\mathbb{R}^N} g(x)|\omega_n(x,0)|^{2_\alpha^*}dx\notag
 \\&\leq \lambda C\|\omega_n\|_{X^\alpha}^q+\displaystyle\int_{\mathbb{R}^N}g(x)|\omega_n(x,0)|^{2_\alpha^*}dx\rightarrow0, \ \text{as}\ n\rightarrow\infty \ \text{and}\ \lambda\rightarrow0.\notag
\end{align}
This yields a contradiction with  Lemma \ref{l5.1}.
Note that
\begin{align}
\frac{\alpha}{N}S^\frac{N}{2\alpha}+\frac{\overline{\delta_0}}{2}&\geq I_{\lambda,g}(\omega)=\displaystyle\sup_{t\geq0}I_{\lambda,g}(t\omega)\notag
 \\&\geq I_{\lambda,g}(t_\omega\omega)
 \nonumber\\&\geq I_\infty^1(t_\omega\omega)-\frac{\lambda}{q}C\displaystyle\int_{\mathbb{R}^N}f(x)|\omega_n(x,0)|^qdx\notag
  \\&\geq I_\infty^1(t_\omega\omega)-\lambda C\|\omega_n\|_{X^\alpha}^q.
\end{align}
Then
\begin{align}\label{e5.44}
I_\infty^1(t_\omega\omega)&\leq \frac{\alpha}{N}S^\frac{N}{2\alpha}+\frac{\overline{\delta_0}}{2}+\lambda C\|\omega_n\|_{X^\alpha}^q.
\end{align}
From (\ref{ee5.41}) and (\ref{e5.44}),
there exists $\overline{\lambda_\delta}>0$ such that for $\lambda\in(0,\overline{\lambda_\delta})$ we have
\begin{equation*}
I_\infty^1(t_\omega\omega)\leq \frac{\alpha}{N}S^\frac{N}{2\alpha}+\overline{\delta_0}.
\end{equation*}
By virtue of Lemma \ref{l5.6}, we obtain that $\Psi(t_\omega\omega)\in M_\delta$ or $\Psi(\omega)\in M_\delta$.
\end{proof}

Let $\overline{c_\lambda}:=\frac{\alpha}{N}S^\frac{N}{2\alpha}-\sigma(\varepsilon_0)$ and
\begin{equation*}
N_{\lambda,g}(\overline{c_\lambda}):=\left\{\omega\in N_{\lambda,g};I_{\lambda,g}(\omega)\leq \overline{c_\lambda} \right\}.
\end{equation*}
Similar to the proof of Lemma \ref{l4.5}, we have

\begin{lemma}\label{l5.8}
If $\omega$ is a critical point of $I_{\lambda,g}$ restricted on $N_{\lambda,g}$, then it is a critical point of $I_{\lambda,g}$ in $X^\alpha(\mathbb{R}^{N+1}_+)$.
\end{lemma}

\begin{lemma}\label{l5.9}
$I_{N_{\lambda,g}}$ satisfies the $(PS)$-condition on $N_{\lambda,g}(\overline{c_\lambda})$, where  $I_{N_{\lambda,g}}$ denotes the restriction of $I_{\lambda,g}$ on $N_{\lambda,g}$.
\end{lemma}
\begin{proof}
Let $\{\omega_n\}\subset N_{\lambda,g}(\overline{c_\lambda})$ be a $(PS)$ sequence. Then there exists a sequence $\{\theta_n\}\subset\mathbb{R}$ such that
\begin{equation*}
 I'_{\lambda,g}(\omega_n)=\theta_n\Psi'_{\lambda,g}(\omega_n)+o_n(1).
\end{equation*}
It follows (\ref{e5.5}) that
\begin{equation*}
\Psi'_{\lambda,g}(\omega_n)\omega_n<0.
\end{equation*}
So there exists a subsequence, still denoted by $\{\omega_n\}$, such that
\begin{equation*}
\Psi'_{\lambda,g}(\omega_n)\omega_n\rightarrow a\ (a\leq0), ~\text{as}~n\rightarrow\infty.
\end{equation*}

If $a=0$,  from (\ref{e5.5}) we get
\begin{align}
0\leftarrow\Psi'_{\lambda,g}(\omega_n)\omega_n&=-2\|\omega_n\|_{X^\alpha}^2+(4-q) \lambda\displaystyle\int_{\mathbb{R}^N}f(x)|\omega_n(x,0)|^qdx\\&~~~~-(2_\alpha^*-4)\displaystyle\int_{\mathbb{R}^N}g(x)|\omega_n(x,0)|^{2_\alpha^*}dx\notag
\\&\leq-2\|\omega_n\|_{X^\alpha}^2\nonumber\\&\leq0.\notag
\end{align}
Then
\begin{equation*}
 \|\omega_n\|_{X^\alpha}\rightarrow0,~\text{as}~n\rightarrow\infty,
\end{equation*}
which yields a contradiction with Lemma \ref{l5.1}. Thus,  $a<0$. Due to $I'_{\lambda,g}(\omega_n)\omega_n=0$,  we deduce that $\theta_n\rightarrow0$ and
\begin{equation*}
I'_{\lambda,g}(\omega_n)\rightarrow0,~\text{as}~n\rightarrow\infty.
\end{equation*}
By virtue of Lemma \ref{l5.3}, we arrive at the desired result.
\end{proof}

We are left to prove Theorem \ref{t1.2}.

\begin{proof}[Proof of Theorem \ref{t1.2}]
Let $\delta,\,\overline{\lambda_\delta}>0$ be given as in Lemmas \ref{l5.6} and \ref{l5.7}.  For each $z\in M$, let $\overline{F}(z)=t_zv_{\varepsilon,z}$. By Lemma \ref{l5.5}, it belongs to $N_{\lambda,g}(\overline{c_\lambda})$.
It follows Lemma \ref{l5.7} that $\Phi(N_{\lambda,g}(\overline{c_\lambda}))\subset M_\delta$ for $\lambda<\overline{\lambda_\delta}$.

Define $\overline{\xi}:[0,1]\times M\rightarrow M_\delta$ by
\begin{equation*}
[0,1]\times M\ni(\theta,z)\longmapsto\overline{\xi}(\theta,z)=\Phi\left(t_zv_{(1-\theta)\varepsilon,z}\right)\in N_{\lambda,g}^-(\overline{c_\lambda}).
\end{equation*}
The straightforward calculations provide that $\overline{\xi}(0,z)=\Phi\circ \overline{F}(z)$ and $\displaystyle\lim_{\theta\rightarrow1^-}\overline{\xi}(\theta,z)=z$.
Hence $\Psi\circ \overline{F}$ is homotopic to $j:M\rightarrow M_\delta$. By virtue of Lemma \ref{l5.9} with Propositions \ref{p4.1} and \ref{p4.2}, we obtain that $I_{N_{\lambda,g}^-(\overline{c_\lambda})}$ has at least $cat_{M_\delta}(M)$ critical points in $N_{\lambda,g}^-(\overline{c_\lambda})$. Based on Lemma \ref{l5.8}, we know that $I_{\lambda,g}$ has at least $cat_{M_\delta}(M)$ critical points  in $N_{\lambda,g}(\overline{c_\lambda})$. Analogous to the proof of \cite[Theorem 5.3]{20}, we can obtain that system (\ref{e2.7}) admits at least $cat_{M_\delta}(M)$ positive solutions  in $X^\alpha(\mathbb{R}^{N+1}_+)$.
\end{proof}

{\bf Acknowledgments}
The third author would like to thank Professor Jinqiao Duan for fruitfulness discussions. He also feels grateful to the
Department of Mathematics at Georgia Institute of Technology for its hospitality and
generous support during his visit since April 2019.

\end{document}